\documentclass[reqno]{amsart}
\usepackage{a4wide}
\usepackage{hyperref}
\usepackage{amsmath,amssymb,amsthm,amsfonts}
\usepackage[switch]{lineno}
\usepackage{xcolor}

\def\p{\partial}
\def\bm{\boldsymbol}
\def\vep{\varepsilon}
\def\q{\quad}
\def\cd{\cdot}
\def\O{\Omega}
\def\na{\nabla}
\def\cal{\mathcal}
\def\f{\frac}

\def\bb{\mathbb}
\def\ba{\begin{equation}\begin{aligned}}
	\def\ea{\end{aligned}\end{equation}}
\def\bn{\[\begin{aligned}}
\def\en{\end{aligned}\]}
\def\l{\label}
\begin{document}
\title[ ]
{On the Lifespan of Axisymmetric Hall-MHD with Swirl}

\author[Z. Wu, L. Yang]
{Zhipeng Wu, Linbin Yang}

\address{ Zhipeng Wu \newline
School of Mathematics and Statistics,
Nanjing University of Information Science and Technology,
Nanjing 210044, China}
\email{wuzhipeng@nuist.edu.cn}

\address{ Linbin Yang \newline
School of Mathematics and Statistics,
Nanjing University of Information Science and Technology,
Nanjing 210044, China}
\email{yanglinbin@nuist.edu.cn}


\begin{abstract}
In this paper, we study the three-dimensional inviscid incompressible resistive Hall-MHD system in the axisymmetric setting with nontrivial swirl velocity and purely azimuthal magnetic. Assuming only that the swirl component of the initial velocity is sufficiently small, we prove that the lifespan of the strong solution can be sufficiently large. An explicit lifespan lower bound in terms of the size of the initial swirl is given. Moreover, we also study the behavior of the lifespan as the resistivity tends to zero.

\end{abstract}

\maketitle
\numberwithin{equation}{section}
\newtheorem{theorem}{Theorem}[section]
\newtheorem{lemma}[theorem]{Lemma}
\newtheorem{proposition}[theorem]{Proposition}
\newtheorem{remark}[theorem]{Remark}
\allowdisplaybreaks

\section{Introduction}

In this paper, we consider the three-dimensional inviscid, incompressible and resistive Hall-HMD equations
\begin{equation} \label{eq:Hall-MHD}
\left\{\begin{aligned}
	& \p_t \bm{u} + \bm{u}\cdot \nabla \bm{u} + \nabla p =\frac{1}{\mu_0}\cdot \bm{h}\cdot \nabla \bm{h},\\
	& \p_t \bm{h} + \bm{u}\cdot \nabla \bm{h} + \nu_0 \nabla \times [(\nabla \times \bm{h}) \times \bm{h} ] = \bm{h} \cdot \nabla \bm{u} + \nu \Delta \bm{h},\\
	& \nabla \cdot \bm{u} =0,\\
	& \nabla \cdot \bm{h} =0,
\end{aligned}\right.
\end{equation}
with the initial data
\begin{equation}\label{eq:ini}
    \bm{u}(0,x)=\bm{u}_0, \q \bm{h}(0,x)=\bm{h}_0.
\end{equation}
 Here $\bm{u}:\mathbb{R}^3 \mapsto \mathbb{R}^3$ represents the velocity and $\bm{h}:\mathbb{R}^3 \mapsto \mathbb{R}^3$ is the magnetic filed. $p:\mathbb{R}^3 \mapsto \mathbb{R}$ is the pressure. Compared with the classical MHD system, the Hall-MHD equations contain the additional nonlinear Hall term $\na\times\left((\na\times \bm{h})\times \bm{h}\right)$, which becomes important in plasma phenomena on length scales below the ion inertial length (see \cite{Somov:2010,Shalybkov}). This term introduces stronger derivative interactions and makes the analysis more delicate than that for the standard MHD equations.
 
 We focus on axially symmetric solutions of system \eqref{eq:Hall-MHD} for $\nu> 0$ (viscous magnetic field). For the proof of the lifespan with fixed positive coefficients, the constants $\mu_0,\nu_0,\nu$ can be normalized. Therefore, except in section \ref{ploi} where the dependence on the magnetic resistivity is explicitly tracked, we set $\mu_0=\nu_0=\nu=1$ for simplicity.
 Most of the proof is carried out in the cylindrical coordinates $(r,\theta,z)$, i.e. for $x=(x_1,x_2,x_3)\in \mathbb{R}^3$,
 \begin{equation*}
 	r=\sqrt{x_1 ^2 + x_2 ^2},\q \theta =\arctan \frac{x_2}{x_1}, \q z=x_3.
 \end{equation*}
 A solution of \eqref{eq:Hall-MHD} is called an axially symmetric solution if and only if
 \begin{equation*}
 \left\{\begin{aligned}
 	& \bm{u} = u_r(t, r, z)\bm{e_r} + u_{\theta}(t, r, z)\bm{e_{\theta}} + u_z(t, r, z)\bm{e_z},\\
 	& \bm{h} = h_{r}(t, r, z)\bm{e_r} + h_{\theta}(t, r, z)\bm{e_{\theta}} + h_z(t, r, z)\bm{e_z},
 \end{aligned}\right.
 \end{equation*}
 satisfy the system \eqref{eq:Hall-MHD}. Here, the basis vectors $\bm{e_r}$, $\bm{e_{\theta}}$, $\bm{e_z}$ are
 \begin{equation*}
 	\bm{e_r} = \left(\frac{x_1}{r}, \frac{x_2}{r}, 0\right), \q \bm{e_{\theta}} = \left(-\frac{x_2}{r}, \frac{x_1}{r}, 0\right) \q \bm{e_z}=\left(0,0,1\right).
 \end{equation*}
 
 Owing to the local well-posedness theory established in \cite{CDL}, system \eqref{eq:Hall-MHD} also admits axially symmetric local classical solutions of the special form:
 \begin{equation*}
     \left\{\begin{aligned}
         &\bm{u}=u_r(t,r,z)\bm{e_r}+u_\theta(t,r,z)\bm{e_\theta}+u_z(t,r,z)\bm{e_z}, \\
 &\bm{h}=h_{\theta}(t,r,z)\bm{e_{\theta}}.
     \end{aligned}\right.
 \end{equation*}
 In this situation, the system \eqref{eq:Hall-MHD} can be written as following
 \begin{equation} \label{eq:Ax_Hall-MHD}
 \left\{\begin{aligned}
 	& \p_t u_r + (u_r\p_r + u_z\p_z)u_r - \frac{u_{\theta} ^2}{r} + \p_r p = -\frac{h_{\theta} ^2}{r},\\
 	& \p_t u_{\theta} + (u_r\p_r + u_z\p_z)u_{\theta} + \frac{u_{\theta}u_{r}}{r} =0,\\
 	& \p_t u_z + (u_r\p_r + u_z\p_z)u_z + \p_z p=0,\\
 	& \p_t h_{\theta} + (u_r\p_r + u_z\p_z)h_{\theta} - \frac{h_{\theta}u_r}{r} =  (\Delta - \frac{1}{r^2})h_{\theta} + \frac{\p_z h_{\theta} ^2}{r},\\
 	& \nabla\cdot \bm{u} = \p_r u_r + \frac{u_r}{r} + \p_z u_z =0.
 \end{aligned}\right.
 \end{equation}
 In this paper, we will prove the lifespan of this family of strong solutions can be arbitrarily large if the initial swirl component of velocity is small enough. Here is the main result:
 \begin{theorem}\label{theo}
Let $(\bm{u},\bm{h})$ be a smooth axially symmetric solution to the initial problem \eqref{eq:Ax_Hall-MHD}, satisfying \eqref{eq:ini}, while the initial data $(\bm{u_0}, \bm{h_0})\in H^3(\bb{R}^3)\times H^3(\bb{R}^3)$, satisfy $\na\cd \bm{u}_0  = h_{r,0}=h_{z,0}=0$. Suppose
\begin{equation*}
	\|\nabla \times ({u}_{0,\theta}\bm{e_{\theta}})\|_{L^\infty} := \vep << 1.
\end{equation*}
Then the solution $(\bm{u},\bm{h})$ keeps in $H^3(\bb{R}^3)\times H^3(\bb{R}^3)$ when $t\leq T_*$, where $T_*$ satisfies:
\begin{equation} \label{res}
    T_*=\frac{C_*}{({1+E_0})^{4/5}}\left(\log\log\log\log(\varepsilon^{-1})\right)^{3/5}.
\end{equation}
Here $C$ is a nonnegative constant and $E_0:=\|(\bm{u}_0,\bm{h}_0)\|_{H^3}^2$.
\end{theorem}
\begin{remark}\label{rea}
 It is natural to investigate how the lifespan lower bound depends on $\nu$, especially in the vanishing resistivity $\nu\to0^+$, see $\eqref{eq:Hall-MHD}_1$. This issue will be discussed in Section \ref{ploi}, where we show the details how the lower bound obtained in \eqref{res} degenerates as $\nu\rightarrow 0^+$.
\end{remark}

The mathematical theory of the Hall-MHD system has attracted considerable attention in recent years. For the general Hall-MHD equations, extensive research has been devoted to issues such as local well-posedness, blow-up criteria, and regularity. In particular, Chae, Degond, and Liu \cite{CDL} established the global existence of weak solutions and the local well-posedness of smooth solutions in the Sobolev space $H^s(\mathbb{R}^3)$ with $s>5/2$. Subsequently, Dai \cite{Dai2020} extended the local well-posedness theory to the $n$-dimensional case in $H^s(\mathbb{R}^n)$ for $n\ge 2$ and $s>n/2$, and later proved the local well-posedness of the Hall-MHD system in optimal Sobolev spaces \cite{Dai2021}. Wang and Zuo \cite{Wang2014} investigated the initial value problem for the Hall-MHD system and obtained a Beale--Kato--Majda type blow-up criterion for smooth solutions with partial viscosity. More recently, Rahman \cite{Rahman2025} established several regularity results for both the 3-D Hall equation and the 3-D Hall-MHD system.

In the axially symmetric setting, the special geometric structure provides additional cancellations and allows one to derive refined a priori estimates. For the Hall-MHD system, global well-posedness for the axisymmetric incompressible viscous andresistive HMHD equations was established by Fan et al. \cite{Fan2013}. Li and Pan \cite{li2022single} established a single-component BKM-type regularity criterion for the inviscid axially symmetric Hall-MHD system with swirl-free current density, showing that the boundedness of the partial vorticity associated with the swirl component of the velocity suffices to prevent blow-up. Later, Li and Yang \cite{LZYM2022} obtained a single-component regularity criterion for the non-resistive axially symmetric Hall-MHD system. These results indicate that the swirl component plays an essential role in the axisymmetric regularity theory of Hall-MHD.

The idea that a small swirl component can lead to a long lifespan is traced back to Danchin \cite{Danc2012}, where the author shows the lifespan $T_*$ of 3D axisymmetric Euler equations satisfies the following lower bound:
\[
T_* \geq \frac{1}{C\left\|w_{0, \theta}\right\|_{B_{\infty, 1}^0}} \log \left(1+\frac{1}{2} \log \left(1+\frac{C\left\|w_{0, \theta}\right\|_{B_{\infty, 1}^0}}{\|\left(v_{0, \theta}\right)^2\|_{B_{\infty, 1}^1}}\right)\right)\,.
\]
 However, the proof in \cite{Danc2012} is valid only on a smooth axisymmetric domain $D$ away from the $z$-axis. Later Li and Zhou \cite{LZZT2024} overcome this limit of the domain, by showing the result is also valid in the whole space. Recently Yang and Zhou \cite{MR5040962} studied the lifespan estimate of ideal Hall-MHD ($\nu=0$) without swirl ($u_\theta=0$). Due to the absence of magnetic resistivity, authors in \cite{MR5040962} had to assume that the initial magnetic field was small. We also refer readers to \cite{HL2023} for lifespan estimates of anisotropic 3D Navier-Stokes equations.

Motivated by above works, we study the lifespan problem for the axially symmetric resistive Hall-MHD system with nontrivial swirl velocity. Our result requires only the initial swirl component of the velocity to be sufficiently small, with no size restriction on the magnetic field. In addition, we also record how the above lifespan depends on the resistive coefficient $\nu$ under the same assumption that the initial swirl component of the velocity is sufficiently small. 

Our proof develops a bootstrap scheme adapted to the Hall structure, and the main theorems is carried out mainly in the following three quantities:
\begin{align*}
		\Omega: =\f{\omega_\theta}{r},\q J=\f{\omega_r}{r}, \q \mathcal{H}: =\frac{h_\theta}{r}.
\end{align*}
Here $\omega_r$, $\omega_\theta$ are components of vorticity $\bm{\omega}$ which is defined by
 \[
 \bm{\omega} = \na\times \bm{u} =\omega_r(t,r,z)\bm{e_r} + \omega_\theta(t,r,z)\bm{e_\theta} + \omega_z (t,r,z)\bm{e_z},
 \]
 where 
 \[
 \omega_r= -\p_zu_{\theta}, \q \omega_{\theta}=\p_zu_r-\p_ru_z, \q \omega_z=\p_ru_\theta+\f{u_\theta}{r}.
 \]
Morever, by the first three equations of \eqref{eq:Ax_Hall-MHD}, the vorticity components $(\omega_r,\omega_{\theta},\omega_z)$ satisfy
\begin{equation} \label{eq:w}
\left\{\begin{aligned}
	& \p_t \omega_r + (u_r\p_r + u_z\p_z)\omega_r = (\omega_r\p_r + \omega_z\p_z)u_r,\\
	& \p_t \omega_{\theta} + (u_r\p_r + u_z\p_z)\omega_{\theta} = \frac{u_r\omega_{\theta}}{r} + \frac{1}{r}\p_z u_{\theta} ^2 -\frac{1}{r}\p_z h_{\theta} ^2,\\
	& \p_t \omega_{z} + (u_r\p_r + u_z\p_z)\omega_z = (\omega_r\p_r + \omega_z\p_z)u_z.
\end{aligned}\right.
\end{equation}

Using \eqref{eq:w}$_{1,2}$ and \eqref{eq:Ax_Hall-MHD}$_4$, it is easy to deduce the above three quantities satisfy
\begin{equation*}
    \left\{\begin{aligned}
	& \p_t \Omega + \bm{b}\cdot \nabla\Omega = -\p_z \mathcal{H}^2	 - 2\frac{u_{\theta}}{r}J,\\
	& \p_t J + \bm{b}\cdot \nabla J = (\omega_r\p_r + \omega_z\p_z)\frac{u_r}{r},\\
    & \p_t\cal{H}+\bm{b}\cd \na\cal{H}-(\Delta+\f2r\p_r)\cal{H}= 2\cal{H}\p_z\cal{H},
\end{aligned}\right.
\end{equation*}
where $\bm{b}=u_r\bm{e_r}+u_z\bm{e_z}$. Here, $\cal{H}$ is not introduced as an additional quantity in the final higher-order energy estimate. Instead, it serves as an intermediate axially symmetric magnetic variable that enables us to obtain the estimate of $\|\nabla \bm{h}\|_{{L^2}(0,t;L^\infty)}$ and $\|\na^2 \bm{h}\|_{L^2(0,t;L^3)}^2$.

\subsection*{Notations}
Through out this paper, $C$ denotes a positive constant whose value may change from line to line, and $C_{a,b,\cdots}$ denotes a positive constant depending only on $a,b,\cdots$. We write $A\lesssim B$ if there exists a constant $C>0$ such that $A\leq CB$, and $A\simeq B$ if both $A\lesssim B$ and $B\lesssim A$ hold.

For $1\leq p\leq \infty$, we denote by $L^p=L^p(\mathbb{R}^3)$ the usual Lebesgue space with norm
\[
\|f\|_{L^p}
:=
\begin{cases}
\left(\displaystyle\int_{\mathbb{R}^3}|f(x)|^p\,dx\right)^{1/p}, & 1\leq p<\infty,\\[1ex]
\operatorname*{ess\,sup}\limits_{x\in\mathbb{R}^3}|f(x)|, & p=\infty.
\end{cases}
\]
For an axially symmetric function $f=f(r,z)$, its $L^p(\mathbb{R}^3)$ norm can also be written as
\[
\|f\|_{L^p}
=
\begin{cases}
\left(2\pi\displaystyle\int_0^\infty\int_{\mathbb{R}} |f(r,z)|^p\, r\,dr\,dz\right)^{1/p}, & 1\leq p<\infty,\\[1ex]
\operatorname*{ess\,sup}\limits_{(r,z)\in \mathbb{R}^+\times\mathbb{R}} |f(r,z)|, & p=\infty.
\end{cases}
\]

For $k\in\mathbb{N}$ and $1\leq p\leq \infty$, $W^{k,p}=W^{k,p}(\mathbb{R}^3)$ denotes the usual Sobolev space with norm
\[
\|f\|_{W^{k,p}}
:=
\sum_{|\alpha|\leq k}\|D^\alpha f\|_{L^p},
\]
and $\dot W^{k,p}=\dot W^{k,p}(\mathbb{R}^3)$ denotes the homogeneous Sobolev space with seminorm
\[
\|f\|_{\dot W^{k,p}}
:=
\sum_{|\alpha|=k}\|D^\alpha f\|_{L^p}.
\]
In particular, we write
\[
H^k:=W^{k,2},\qquad \dot H^k:=\dot W^{k,2}.
\]

We will frequently use the following axisymmetric quantities:
\[
\bm{b}:=u_r e_r+u_z e_z,\qquad
\cal{H}:=\frac{h_\theta}{r},\qquad
\Omega:=\frac{\omega_\theta}{r},\qquad
J:=\frac{\omega_r}{r}.
\]

The commutator of two operators $\mathcal{A}$ and $\mathcal{B}$ is defined by
\[
[\mathcal{A},\mathcal{B}]:=\mathcal{A}\mathcal{B}-\mathcal{B}\mathcal{A}.
\]

The rest of this paper is organized as follows. In Section \ref{preli}, we collect several preliminary lemmas that will be used in the proof. Section \ref{proof} is devoted to the derivation of the key a priori estimates and is ready for the proof of the main theorem. Section \ref{endd} is the final proof of the main theorem. Section \ref{ploi} gives the outline of proof of Remark \ref{rea}.

\section{Preliminaries}\label{preli}

 We begin by presenting several auxiliary lemmas that will be used repeatedly in the proof of the main theorem. The first is the classical Gagliardo–Nirenberg interpolation inequality in $\bb{R}^3$, the proof of which can be found in \cite{Nirenberg1959}.

\begin{lemma}[Gagliardo-Nirenberg] \label{ineq:G_N}
	Fix $q, r \in[1, \infty]$ and $j, m \in \mathbb{N} \cup\{0\}$ with $j \leq m$. Suppose that $f\in L^q \cap \dot{W}^{m, r}$ and there exists a real number $\alpha \in[j / m, 1]$ such that
	\begin{align*}
		\frac{1}{p}=\frac{j}{3}+\alpha\left(\frac{1}{r}-\frac{m}{3}\right)+\frac{1-\alpha}{q}.
	\end{align*}
	Then $f \in \dot{W}^{j, p}$ and there exists a constant $C>0$ such that
	\begin{align*}
		\left\|\nabla^j f\right\|_{L^p} \leq C\left\|\nabla^m f\right\|_{L^r}^\alpha\|f\|_{L^q}^{1-\alpha},
	\end{align*}
except the following two cases: 

(i) $j=0, m r<3$ and $q=\infty$; (In this case it is necessary to assume also that either $f \rightarrow 0$ at infinity, or $f \in L^s$ for some $s<\infty$.)
		
(ii) $1<r<\infty$ and $m-j-3 / r \in \mathbb{N}$. (In this case it is necessary to assume also that $\alpha<1$.)
\end{lemma} \qed

We are now in a position to state the following lemma, which provides the fundamental estimates for the system \eqref{eq:Ax_Hall-MHD} in the case of $\nu =1$.

\begin{lemma}\label{ineq:Lp_con}
	Define $\mathcal{H}:=\frac{h_{\theta}}{r}$. Let $(\bm{u},\bm{h})\in H^3(\bb{R}^3)\times H^3(\bb{R}^3)$ be the solution of \eqref{eq:Ax_Hall-MHD}, then we have:
    
    (i) For $p\in[2,\infty]$ and $t\in(0,\infty)$,
	\begin{flalign}	\label{ineq:na_H_L2}
		&\|\mathcal{H}(t,\cdot)\|_{L^p}^p + p(p-1) \int_0^t \int_{\mathbb{R}^3} | \nabla \mathcal{H}(s,x) |^2 | \mathcal{H}(s,x) |^{p-2} dxds \leq \|\mathcal{H}_0\|_{L^p}^p .\\
        \label{ineq:H_Lp}
		&\|\mathcal{H}(t,\cdot)\|_{L^p} \leq \|\mathcal{H}_0\|_{L^p}.
	\end{flalign}
	(ii) For $\bm{u}_0, \bm{h}_0 \in L^2$ and $t \in (0,\infty)$,
	\begin{equation}\label{ineq:uh}
		\|(\bm{u},h_{\theta})(t,\cdot)\|_{L^2}^2 + \int_{0}^{t}\|\nabla h_\theta(s,\cdot)\|_{L^2}^2 ds + \int_{0}^{t}\|\frac{h_\theta}{r} (s,\cdot)\|_{L^2}^2 ds \leq C_0.
	\end{equation}
where $C_0$ depends only on $\|(\bm{u_0},\bm{h_0)}\|_{L^2}$.
\end{lemma}

\begin{proof}
	By $\eqref{eq:Ax_Hall-MHD}_{4}$, we have $\mathcal{H}$ satisfies
\begin{equation} \label{eq:H}
	\p_t \mathcal{H} + (u_r\p_r + u_z\p_z)\mathcal{H} - (\Delta + \frac{2}{r} \p_r) \mathcal{H} - 2 \mathcal{H}\p_z \mathcal{H} = 0.
\end{equation}
Performing the $L^P$ estimate of \eqref{eq:H}

\begin{equation*}
\begin{aligned}
    &\frac{1}{p}\frac{d}{dt}\|\mathcal{H}(t,x)\|_{L^p}^p + (p-1)\int_{\mathbb{R}^3} |\mathcal{H}(t,x)|^{p-2} |\nabla\mathcal{H}(t,x)|^2 dx  \\
    &\qquad =\frac{2}{p}\int_{\mathbb{R}^3}\frac{1}{r} \p_r |\mathcal{H}(t,x) |^p dx + \frac{2}{p+1}\int_{\mathbb{R}^3} \p_z\big(\mathcal{H}(t,x) |\mathcal{H}(t,x) |^p\big) dx\\
    &\qquad  =\frac{4\pi}{p} \int_{\mathbb{R}}\int_0^{\infty} \p_r |\mathcal{H}(t,r,z)|^p dr dz\\
    &\qquad =-\frac{4\pi}{p} \int_{\mathbb{R}} |\mathcal{H}(t,0,z)|^p dz \leq 0.
\end{aligned}
\end{equation*} 
Integrating over $(0,t)$, one derives \eqref{ineq:na_H_L2}. The inequality \eqref{ineq:H_Lp} follows by letting $p \rightarrow \infty$. Meanwhile, the inequality \eqref{ineq:uh} follows from the standard $L^2$ estimate of the system \eqref{eq:Hall-MHD}.
\end{proof}
The following lemmas are the estimates for the heat equation. The first one is the maximal $L_t^qL^p$-regularity for the heat flow, reader can find the proof in Theorem 7.3 of \cite{Lema2002}. For the second redear can find the proof in \cite{Robbin} ( Appendix D.).
\begin{lemma}\label{ineq:heat}
    Let us define the operator $\mathcal{A}$ by the formula
\begin{equation*}
    \mathcal{A} :\q f \longrightarrow \int_{0}^{t} \nabla^2e^{(t-s)\Delta} f(s,\cdot)ds.
\end{equation*}
Then $\mathcal{A}$ is bounded from $L^q(0,T;L^p\big({\mathbb{R}^d})\big)$ to itself for every $T \in (0,\infty]$ and $1<p,q<\infty$.  Moreover, there holds:
\begin{equation*}
    \|\mathcal{A}f\|_{L^q(0,T;L^p)} \leq C\|f\|_{L^q(0,T;L^p)}.
\end{equation*}
\end{lemma}\qed
\begin{lemma}\label{BXXL}
Let $1\le l\le r\le \infty$ and let $\alpha$ be a multi-index. For any $g\in L^l(\bb{R}^3)$ and $t>0$, one has
\[
\|\p^\alpha e^{t\Delta}g\|_{L^r(\bb{R}^3)}
\le C\,t^{-\f{|\alpha|}{2}-\frac{3}{2}\left(\f1l-\f1r\right)} \|g\|_{L^l(\bb{R}^n)}.
\]
In particular, for $l=r=2$,
\[
\|e^{t\Delta}g\|_{L^2} \le C\|g\|_{L^2}.
\]
\end{lemma}\qed

The following lemma is a direct consequence of the Biot–Savart law and the $L^p$-boundedness of Calderón–Zygmund singular integral operators; a detailed proof can be found in \cite{MR1902055, MR3030713}.
\begin{lemma}\label{ineq:nau_b}
    Let $\bm{u} = u_r\bm{e_r} + u_\theta\bm{e_\theta} + u_z\bm{e_z}$ be an axially symmetric vector field, $\bm{\omega} = \nabla \times \bm{u} = \omega_r\bm{e_r} + \omega_\theta\bm{e_\theta} + \omega_z\bm{e_z}$ and $\bm{b}=u_r\bm{e_r}+u_z\bm{e_z}$. Then we have
\begin{equation*}
    \|\nabla\bm{u}\|_{L^p} \leq C_p\|\bm{\omega}\|_{L^p},
\end{equation*}
and
\begin{equation*}
    \|\nabla\bm{b}\|_{L^p} \leq C_p\|\omega_\theta\|_{L^P},
\end{equation*}
for all $1<q<\infty$.
\end{lemma}\qed

The following lemma shows that the $L^p$, $1<p<\infty$, norm of $\na\f{u_r}{r}$ is controlled by the $L^p$ norm of $\Omega$. It's proof can be found in \cite{Lei2015}( equation (A.5) ).
\begin{lemma} \label{ineq:nau_rr}
 	Define $\Omega := \frac{\omega_{\theta}}{r}$. For $1< p < \infty$, there exists an absolute constant $C_p>0$ such that
	\begin{equation*}
		\|\nabla\frac{u_r}{r}(t,\cdot)\|_{L^p} \leq C_p \|\Omega(t,\cdot)\|_{L^p}.
	\end{equation*}
\end{lemma} \qed

The following lemma gives the uniform bound to the transport equation whose velocity field is given by the divergence-free vector $\tilde{\bm{b}}:=v_r\bm{e_r}+v_z\bm{e_z}$.
\begin{lemma} \label{ineq:uLp}
	Let $\bm{v}(t,\cdot) : \mathbb{R} ^3 \mapsto \mathbb{R} ^m$, whose initial date $\bm{v_0}\in L^p$ with $1\leq p\leq \infty$ and $m\in \mathbb{N}$, be a weak solution to
	\begin{equation*}
		\p_t \bm{v} + \tilde{\bm{b}}\cdot \nabla\bm{v} = \bm{A}\cdot \bm{v} + \bm{F}.
	\end{equation*}
	Here $\bm{A}(t,\cdot): \mathbb{R}^3 \mapsto \mathbb{R}^{m \times m}$ is uniformly bounded, while $\bm{F}(t,\cdot): \mathbb{R}^3 \mapsto \mathbb{R}^m$ belongs to $L^p$. Then $\bm{v}$ satisfies
	\begin{equation*}
		\|\bm{v}(t,\cdot)\|_{L^p} \leq \|\bm{v_0}\|_{L^p} + \int_0^t \left( \|\bm{A}(s,\cdot)\|_{L^\infty}\|\bm{v}(s,\cdot)\|_{L^\infty} + \|\bm{F}(s,\cdot)\|_{L^p} \right)ds,
	\end{equation*}
	and
	\begin{equation*}
		\|\bm{v}(t,\cdot)\|_{L^p} \leq \left( \|\bm{v}_0\|_{L^p} + \int_0^t \|\bm{F}(s,\cdot)\|_{L^p} ds \right) \exp\left( \int_0^t \|\bm{A}(s,\cdot)\|_{L^\infty} ds \right).
	\end{equation*}
\end{lemma} \qed

Lastly we focus on the following two lemmas, the first one is the estimation of the triple product form with commutators. The second is logarithm inequality.
\begin{lemma}\label{ineq:commu}
Let $m \in \mathbb{N}$ and $m \geq 2, \boldsymbol{f}, \boldsymbol{g}, \boldsymbol{k} \in C_0^{\infty}\left(\mathbb{R}^3\right)$. The following estimates hold:
$$
\left|\int_{\mathbb{R}^3}\left[\nabla^m, \boldsymbol{f} \cdot \nabla\right] \boldsymbol{g} \nabla^m \boldsymbol{k} d x\right| \leq C\left\|\nabla^m(\boldsymbol{f}, \boldsymbol{g}, \boldsymbol{k})\right\|_{L^2}^2\|\nabla(\boldsymbol{f}, \boldsymbol{g})\|_{L^{\infty}} .
$$
\end{lemma}
\begin{proof}
    We apply the H\"older inequality, one derives
$$
\left|\int_{\mathbb{R}^3}\left[\nabla^m, \boldsymbol{f} \cdot \nabla\right] \boldsymbol{g} \nabla^m \boldsymbol{k} d x\right| \leq\left\|\left[\nabla^m, \boldsymbol{f} \cdot \nabla\right] \boldsymbol{g}\right\|_{L^2}\left\|\nabla^m \boldsymbol{k}\right\|_{L^2}.
$$
Due to the commutator estimate by Kato-Ponce\cite{MR951744}, it follows that
$$
\left\|\left[\nabla^m, \boldsymbol{f} \cdot \nabla\right] \boldsymbol{g}\right\|_{L^2} \leq C\left(\|\nabla \boldsymbol{f}\|_{L^{\infty}}\left\|\nabla^m \boldsymbol{g}\right\|_{L^2}+\|\nabla \boldsymbol{g}\|_{L^{\infty}}\left\|\nabla^m \boldsymbol{f}\right\|_{L^2}\right) .
$$
\end{proof}

\begin{lemma}[See \cite{Li2022}, Corollary 2.8] \label{ineq:BMO}
	For any divergence free vector field $\boldsymbol{\boldsymbol{g}}: \mathbb{R}^3 \rightarrow \mathbb{R}^3$ such that $\boldsymbol{\boldsymbol{g}} \in 
    H^3\left(\mathbb{R}^3\right)$, the following estimate holds:
    \begin{equation*}
    \begin{aligned}
	\|\nabla \boldsymbol{\boldsymbol{g}}\|_{L^{\infty}\left(\mathbb{R}^3\right)} \lesssim 1+\|\nabla \times \boldsymbol{\boldsymbol{g}}\|_{B M O\left(\mathbb{R}^3\right)} \log \left(e+\|\boldsymbol{\boldsymbol{g}}\|_{H^3\left(\mathbb{R}^3\right)}\right).
	\end{aligned}
    \end{equation*}
\end{lemma}	\qed

\section{The proof of Theorem 1.1}\label{proof}
This section is devoted to the proof of Theorem \ref{theo}. Applying Lemma \eqref{ineq:uLp} to $\eqref{eq:w}_{1,3}$, we get 
\begin{equation} \label{ineq:prio1}
\begin{aligned}
	\|\nabla \times (u_{\theta}\bm{e_{\theta}}) (t,\cdot)\|_{L^\infty} &= \|(\omega_r,\omega_z)(t,\cdot)\|_{L^\infty} \\
	& \leq \|\nabla \times (u_{0,\theta}\bm{e_{\theta}})\|_{L^\infty}\exp\left( \int_0^t \|\nabla \bm{u}(s,\cdot)\|_{L^\infty} ds \right)\\
	& \leq \vep \exp\left( C\int_0^t \|\bm{u}(s,\cdot)\|_{H^3} ds\right).
\end{aligned}
\end{equation}
By local well-posedness of Hall-MHD (readers can refer to \cite{li2025}), there exists $T_0 = C_0(\|(\bm{u_0}, \bm{h_0}\|_{H^3})$ such that
\begin{equation*}
	\|(\bm{u},\bm{h})(t,\cdot)\|_{H^3} \leq C\|(\bm{u_{0}},\bm{h_0})\|_{H^3},\q \text{for all } t\in[0,T_0).
\end{equation*}
And for any $t\in[0,T_0)$, the estimate \eqref{ineq:prio1} indicates 
\begin{equation*}
	\|\nabla \times (u_{\theta}\bm{e_{\theta}})\|_{L^\infty} \leq \vep \exp\left( C\|\bm{u_0}\|_{H^3} T_0 \right) \leq \vep \exp\left( CC_0 \right), \q \text{for all } t\in[0,T_0).
\end{equation*}
Thus there exists $T_*\geq \min\{T_0, \vep^{-1}\exp\left( -CC_0 \right)\} > 0$, such that
\begin{equation}\label{ineq:prio2}
	t\sup_{0\leq s \leq t} \|\nabla \times (u_{\theta}\bm{e_{\theta}})(s,\cdot)\|_{L^\infty} \leq 1, \q\text{for all } t\in[0,T_*).
\end{equation}
The following argument is carried before the aforementioned $T_*$.

We begin with the following $\|\O\|_{L^\infty_{T_*}(L^2\cap L^6)}$ and $\|\mathcal{H}\|_{L^\infty_{T_*}\dot{H}^1\cap L^2_{T_*}\dot{H}^2}$ estimates. Here goes the proposition. 

\begin{proposition} \label{OH2}
    Define $\Omega:=\frac{\omega_\theta}{r}$ and $\cal{H}:=\frac{h_\theta}{r}$. Assume that $\nabla \cdot \bm{u}_0=h_0^r=h_0^z \equiv 0$. Let $(\bm{u}, \bm{h})$, satisfying \eqref{ineq:prio2}, be the unique local axially symmetric solution of \eqref{eq:Hall-MHD} with the initial data $\left(\bm{u}_0, \bm{h}_0\right) \in H^3\left(\mathbb{R}^3\right)$. The following $\left(L_{T_*}^{\infty}\left(L^2 \cap L^6\right)\right) \times\left(L_{T_*}^{\infty} \dot{H}^1 \cap L_{T_*}^2 \dot{H}^2\right)$ estimate of $(\Omega, \cal{H})$ holds for any $0<t\leq T_*$
    \begin{equation}\label{EP3211}
    \begin{aligned}
        &\sup _{0 \leq s \leq t}\left(\|\Omega(s, \cdot)\|_{L^2 \cap L^6}^2+\|\nabla \cal{H}(s, \cdot)\|_{L^2}^2\right)+\int_0^{t}\left\|\nabla^2 \cal{H}(s, \cdot)\right\|_{L^2}^2 d s\\
         & \qquad \qquad \leq \exp\big(C(1+E_0)^{4/3}(1+t)^{5/3}\big).
        \end{aligned}
    \end{equation}
    Here $C>0$ is an absolute constant.
\end{proposition}

\begin{proof}
Firstly we show the $L^2$ and $L^6$ estimate of $(\Omega,J)(t,\cd)$. By \eqref{eq:w}$_{1,2}$, one can deduce the couple $\Omega = \frac{{\omega}_\theta}{r}, J = \frac{{\omega}_r}{r}$ satisfying: 
\begin{equation} \label{eq:omega}
\left\{\begin{aligned}
	& \p_t \Omega + \bm{b}\cdot \nabla\Omega = -\p_z \mathcal{H}^2	 - 2\frac{u_{\theta}}{r}J,\\
	& \p_t J + \bm{b}\cdot \nabla J = (\omega_r\p_r + \omega_z\p_z)\frac{u_r}{r}.
\end{aligned}\right.
\end{equation}
Performing the $L^p$ estimates for \eqref{eq:omega}$_1$ and using Lemma \ref{ineq:Lp_con}, one arrives
\begin{equation} \label{ineq:Omega}
\begin{aligned}
	\|\Omega(t,\cdot)\|_{L^p} &\lesssim \|\Omega_0\|_{L^p} + \int_0^t \|\p_z \mathcal{H}^2(s,\cdot)\|_{L^p} ds + \int_0^t \|\frac{u_{\theta}}{r} (s,\cdot)\|_{L^\infty} \|J(s,\cdot)\|_{L^p}ds\\
	& \lesssim \|\Omega_0\|_{L^p} + \|\mathcal{H}_0\|_{L^\infty} \int_0^t \|\p_z \mathcal{H}(s,\cdot)\|_{L^p}ds  +\int_0^t \|\nabla \times (u_{\theta}\bm{e_{\theta}})(s,\cdot)\|_{L^\infty}\|J(s,\cdot)\|_{L^p}ds.\\
\end{aligned}
\end{equation}
Similarly one can deduce from $\eqref{eq:omega}_2$ that
\begin{equation} \label{ineq:J}
\begin{aligned}
	\|J(t,\cdot)\|_{L^p} &\lesssim \|J_0\|_{L^p} + \int_0^t \|(\omega_r, \omega_z)(s,\cdot)\|_{L^\infty}\|\nabla \frac{u_r}{r} (s,\cdot)\|_{L^p} ds \\
	& \lesssim \|J_0\|_{L^p} + \int_0^t \|\nabla \times (u_{\theta}\bm{e_{\theta}})(s,\cdot)\|_{L^\infty}\|\Omega(s,\cdot)\|_{L^p} ds.
\end{aligned}
\end{equation}
Combining \eqref{ineq:Omega} and \eqref{ineq:J}, and using the Gr\"onwall inequality, one derives that
\begin{equation} \label{ineq:O_J}
\begin{aligned}
	\|(\Omega,J)(t,\cdot)\|_{L^p} \lesssim & \left( \|(\Omega_0,J_0)\|_{L^p} + \|\mathcal{H}_0\|_{L^\infty}\int_0^t\|\p_z\mathcal{H}(s,\cdot)\|_{L^p} ds \right) \\
&\times \exp\left( \int_0^t \|\nabla\times (u_{\theta}\bm{e_{\theta}})(s,\cdot)\|_{L^\infty} ds \right)\\
	\lesssim & \|(\Omega_0,J_0)\|_{L^p} + \|\mathcal{H}_0\|_{L^\infty}\int_0^t \|\p_z \mathcal{H}(s,\cdot)\|_{L^p} ds,
\end{aligned}
\end{equation}
for any $t\leq T_*$. 

For $p = 2$, we find \eqref{ineq:O_J} together with the fundamental estimate \eqref{ineq:Lp_con} indicates
\begin{equation}\label{ineq:J_L2}
\begin{aligned}
	\|(\Omega,J)(t,\cdot)\|_{L^2}^2 & \lesssim  \|(\Omega_0,J_0)\|_{L^2}^2 + t\|\mathcal{H}_0\|_{L^\infty}^2 \int_0^t \|\p_z \mathcal{H}(s,\cdot)\|_{L^2}^2 ds\\
	& \lesssim \|(\Omega_0,J_0)\|_{L^2}^2 + t\|\mathcal{H}_0\|_{L^2}^2 \|\mathcal{H}_0\|_{L^\infty}^2 \\ 
    & \lesssim (1+E_0)^2(1+t),
\end{aligned}
\end{equation}
holds for any $t\leq T_*$, which is already a self-closed a priori estimate. Here $E_0:= \|(\bm{u}_0,\bm{h}_0)\|_{H^3}^2$.

For $p=6$. Using Lemma \ref{ineq:G_N}, Lemma \ref{ineq:Lp_con} and the Cauchy-Schwarz inequality, \eqref{ineq:O_J} indicates 
\begin{equation}\label{ineq:O_J_L6}
   \begin{aligned}
        \|(\Omega,J)(t,\cd)\|_{L^6}^2\lesssim&\|(\Omega_0,J_0)\|_{L^6}^2 + \|\mathcal{H}_0\|_{L^\infty}^2t\int_0^t \|\p_z \mathcal{H}(s,\cdot)\|_{L^6}^2 ds\\
        \lesssim &E_0(1+t)\int_0^t\|\na^2\cal{H}(s,\cd)\|_{L^2}^2ds.
   \end{aligned}
\end{equation}

Next, we further deduce its explicit estimate with respect to $t$. Denoting 
$$G(t):=1+\|\na\cal{H}(t,\cdot)\|_{L^2}^2+\int_0^t\|\na^2\cal{H}(s,\cd)\|_{L^2}^2ds.$$
In view of \cite[(3.19)]{li2022single}, we see
\begin{equation}\label{ineq:Grn1}
    \f{d}{dt}\|\na\cal{H}(t,\cd)\|_{L^2}^2 + \|\na^2\cal{H}(t.\cd)\|_{L^2}^2\lesssim\|\na\cal{H}(t,\cd)\|_{L^2}^2 (1+\|\cal{H}_0\|_{L^2}^2 + \|\Omega(t,\cd)\|_{L^2}^{4/3}) + \|\na \bm{h}(t,\cd)\|_{L^2}^2\|\Omega(t,\cd)\|_{L^6}^2.
\end{equation}
Thus, inserting \eqref{ineq:J_L2} and \eqref{ineq:O_J_L6} into \eqref{ineq:Grn1}, one finds
\begin{equation*}
    G^\prime(t)\lesssim \left(1+(1+E_0)^{4/3}(1+t)^{2/3}+E_0(1+t)\|\na \bm{h}(t,\cd)\|_{L^2}^2\right)G(t).
\end{equation*}
Using the G\"ornwall inequality, and by Lemma \ref{ineq:Lp_con}, one can deduce
\begin{equation*}
    G(t)\leq G(0) \exp\big(C(1+E_0)^{4/3}(1+t)^{5/3}\big).
\end{equation*}
This indicates 
\[
\|\na\cal{H}(t,\cdot)\|_{L^2}^2+\int_0^t\|\na^2\cal{H}(s,\cd)\|_{L^2}^2ds\leq(1+E_0) \exp\big(C(1+E_0)^{4/3}(1+t)^{5/3}\big).
\]
Substituting this in \eqref{ineq:J_L2} and \eqref{ineq:O_J_L6}, one concludes \eqref{EP3211}.

\end{proof}

The remaining proposition aims to establish the estimate of $\|\bm{h}(t,\cd)\|_{L^2(0,t;L^\infty)}$. We first estimate $\|h_{\theta}(t,\cd)\|_{L^p}$, then derive the $L^2\cap L^6$ bound for $\na\bm{b}$, the $L_t^1L^\infty$ bound for $\p_z\cal{H}$, and the $L^\infty$ bound for $\omega_\theta$. This ultimately leads to the desired estimate.

\begin{proposition}\label{ineq:h_the}
    Define $\Omega:=\frac{\omega_\theta}{r}$ and $\cal{H}:=\frac{h_\theta}{r}$. Assume that $\nabla \cdot \bm{u}_0=h_{r,0}=h_{z,0} \equiv 0$. Let $(\bm{u}, \bm{h})$, satisfying \eqref{ineq:prio2}, be the unique local axially symmetric solution of \eqref{eq:Hall-MHD} with the initial data $\left(\bm{u}_0, \bm{h}_0\right) \in H^3\left(\mathbb{R}^3\right)$. The following $L^p(2\leq p\leq \infty)$ estimate of $h_\theta$ holds for any $0<t\leq T_*$:
    \begin{equation*}
    \|h_\theta(t,\cdot)\|_{L^p}\leq \exp \left(\exp\big(C(1+E_0)^{4/3}(1+t)^{5/3}\big)\right).
\end{equation*}
Here $C$ is a constant.
\end{proposition}

\begin{proof}
    For any $p \geq 1$, multiplying $h_\theta|h_\theta|^{p-2}$ on \eqref{eq:Ax_Hall-MHD}$_4$, one derives
\begin{equation*}
\begin{aligned}
    \frac{1}{p}\frac{d}{dt}\|h_\theta(t,\cdot)\|_{L^p}^p &\leq \|\frac{u_r}{r}(t,\cdot)\|_{L^{\infty}} \|h_\theta(t,\cdot)\|_{L^p}^p - \int_{\mathbb{R}^3}\frac{|h_\theta|^{p}}{r^2} dx - (p-1)\int_{\mathbb{R}^3}|\nabla h_\theta|^2|h_\theta|^{p-2} dx\\
    &\qquad + \int_{\mathbb{R}^3}\frac{1}{r}\p_z(h_\theta)^2 h_\theta |h_\theta|^{p-2} dx\\
    &\leq \left\|\frac{u_r}{r}(t,\cdot)\right\|_{L^{\infty}} \|h_\theta(t,\cdot)\|_{L^p}^p.
\end{aligned}
\end{equation*}
Here, we have applied the identity
\begin{equation*}
    \int_{\mathbb{R}^3}\frac{1}{r}\p_z(h_\theta)^2 h_\theta |h_\theta|^{p-2} dx = \frac{2}{p+1}\int_{\mathbb{R}^3}\p_z\left(\frac{1}{r}h_\theta |h_\theta|^p\right)dx = 0.
\end{equation*}
Canceling $\|h_\theta(t,\cdot)\|_{L^p}^{p-1}$ on each sides and by the Gr\"onwall inequality, one finds
\ba\l{EHTH}
    \|h_\theta(t,\cdot)\|_{L^p} \leq \|h_{\theta,0}\|_{L^p}  \exp\left(\int_{0}^{t} \|\frac{u_r}{r} (s,\cdot)\|_{L^\infty} ds\right).
\ea
Noting that $\frac{u_r}{r}$ satisfied by Lemma \ref{ineq:G_N} and Lemma \ref{ineq:nau_b}
\begin{equation*}
\begin{aligned}
    \|\frac{u_r}{r}(t,\cdot)\|_{L^\infty} &\lesssim \|\frac{u_r}{r}(t,\cdot)\|_{L^6}^{\frac{1}{2}} \|\nabla \frac{u_r}{r}(t,\cdot)\|_{L^6}^{\frac{1}{2}} \lesssim \|\nabla \frac{u_r}{r}(t,\cdot)\|_{L^2}^{\frac{1}{2}}\|\nabla \frac{u_r}{r}(t,\cdot)\|_{L^6}^{\frac{1}{2}}\\
	& \lesssim \|\Omega(t,\cdot)\|_{L^2}^{\frac{1}{2}} \|\Omega(t,\cdot)\|_{L^6}^{\frac{1}{2}} ,
\end{aligned}
\end{equation*}
for any $t \leq T_*$. Integrating with time over $(0,t)$ and using \eqref{EP3211}, one has
\ba\l{EURR}
\int_0^t\|\frac{u_r}{r}(s,\cdot)\|_{L^\infty}ds\leq C(1+E_0)(1+t)^{3/2}\exp\big(C(1+E_0)^{4/3}(1+t)^{5/3}\big).
\ea
Substituting \eqref{EURR} in \eqref{EHTH}, one concludes
\begin{equation}\label{ESSS}
\begin{aligned}
    \|h_\theta(t,\cdot)\|_{L^p}\leq \exp \left(\exp\big(C(1+E_0)^{4/3}(1+t)^{5/3}\big)\right),
\end{aligned}
\end{equation}
uniformly for $p \in [2,\infty)$, $0\leq t\leq T_*$ and the $L^{\infty}$ estimate of $h_\theta$ is achieved by choosing $p \rightarrow \infty$ in \eqref{ineq:h_the} since the far right above is independent of $p$.
\end{proof}

\begin{proposition}\label{ineq:na_b_Lp}
    Assume that $\nabla \cdot \bm{u}_0 = {h}_{r,0} = h_{z,0} \equiv 0$. Let $(\bm{u},\bm{h})$, satisfying \eqref{ineq:prio2}, be the unique local axially symmetric solution of \eqref{eq:Hall-MHD} with the initial date $(u_0,h_0) \in H^3(\mathbb{R}^3) $. The following $L^p$ estimate of $\nabla \bm{b}$ holds for any $t\geq 0$:
\begin{equation*}
    \|\nabla\bm{b}(t,\cd)\|_{L^p} \leq  \exp \left(\exp\big(C(1+E_0)^{4/3}(1+t)^{5/3}\big)\right).
\end{equation*}
Here $C$ is a constant and $2\leq p\leq 6$.
\end{proposition}
\begin{proof}
    To obtain the $L^p$ estimate of $\nabla \bm{b}$ , we begin by deriving an $L^p$ estimate of $\frac{u_\theta}{r}$ for $2 \leq p \leq \infty$. Due to \eqref{eq:Ax_Hall-MHD}$_2$, one can deduce that $\frac{u_\theta}{r}$ satisfies
\begin{equation*}
    \p_t \frac{u_\theta}{r} +(\bm{b}\cdot\nabla) \frac{u_\theta}{r} + 2\frac{u_r}{r}\cdot\frac{u_\theta}{r} = 0,
\end{equation*}
for any $p\geq1$, multiplying $\frac{u_\theta}{r}|\frac{u_\theta}{r}|^{p-2}$ on both sides and integrating on $\mathbb{R}^3$, one arrives
\begin{equation*}
    \frac{1}{p}\frac{d}{dt}\|\frac{u_\theta}{r}\|_{L^p}^p + {\int_{\mathbb{R}^3} \left(\bm{b}\cdot\nabla\right) \frac{u_\theta}{r} \left|\frac{u_\theta}{r}\right|^{p-2}\frac{u_\theta}{r}dx} + 2\int_{\mathbb{R}^3} \frac{u_r}{r}\left|\frac{u_\theta}{r}\right|^p dx=0.
\end{equation*}
Noticing that
\begin{equation*}
\begin{aligned}
        & {\int_{\mathbb{R}^3} \left(\bm{b}\cdot\nabla\right) \frac{u_\theta}{r} \left|\frac{u_\theta}{r}\right|^{p-2}\frac{u_\theta}{r}dx} = \frac{1}{p}\int_{\mathbb{R}^3} \bm{b}\cdot\nabla\left(\left|\frac{u_\theta}{r}\right|^p\right)dx = 0,\\
\end{aligned}
\end{equation*}
and
\begin{equation*}
\begin{aligned}
        &  \Big|\int_{\mathbb{R}^3} \frac{u_r}{r}\left|\frac{u_\theta}{r}\right|^p dx\Big| \leq \left\|\frac{u_r}{r}\right\|_{L^{\infty}}\left\|\frac{u_\theta}{r}\right\|_{L^p}^p,
\end{aligned}
\end{equation*}
one derives
\[
\frac{1}{p}\frac{d}{dt}\|\frac{u_\theta}{r}\|_{L^p}^p\leq2\left\|\frac{u_r}{r}\right\|_{L^{\infty}}\left\|\frac{u_\theta}{r}\right\|_{L^p}^p.
\]
Canceling $\|\frac{u_\theta}{r}\|_{L^p}^{p-1}$ on each side and using the Gr\"onwall inequality, one finds
\ba\l{EUTHR}
	\|\frac{u_\theta}{r}(t,\cdot)\|_{L^p} \leq& \|\frac{u_{\theta,0}}{r}\|_{L^p} \exp\left(2 \int_{0}^{t} \|\frac{u_r}{r}(s,\cdot)\|_{L^\infty} ds\right) \\
\leq &E_0^{1/2}\exp \left((1+E_0)(1+t)^{3/2}\exp\big(C(1+E_0)^{4/3}(1+t)^{5/3}\big)\right)\\
\leq & \exp \left(\exp\big(C(1+E_0)^{4/3}(1+t)^{5/3}\big)\right),
\ea
for any $p \in [2,\infty]$, $t\leq T_*$.

Next, for any $2\leq p\leq 6$, multiplying $|\omega_\theta|^{p-2}\omega_\theta$ on both sides of the following equation and integrating over $\mathbb R^3$,
\begin{equation*}
    \p_t\omega_\theta + (u_r\p_r + u_z\p_z)\omega_\theta = \frac{u_r}{r}\omega_\theta - \frac{2}{r}u_\theta\omega_r - \p_z\mathcal{H}h_\theta.
\end{equation*}
One derives
\begin{equation}
\begin{aligned}\label{eq:wp-basic}
\frac{1}{p}\frac{d}{dt}\|\omega_\theta(t,\cd)\|_{L^p}^p & +\underbrace{\int_{\mathbb R^3}(u_r\p_r+u_z\p_z)\omega_\theta|\omega_\theta|^{p-2}\omega_\theta dx}_{K_1} \\
& = \underbrace{\int_{\mathbb R^3}\frac{u_r}{r}|\omega_\theta|^p dx}_{K_2} -2\underbrace{\int_{\mathbb R^3}\frac{u_\theta}{r}\omega_r |\omega_\theta|^{p-2}\omega_\theta dx}_{K_3} -\underbrace{\int_{\mathbb R^3}\p_z\mathcal{H} h_\theta |\omega_\theta|^{p-2}\omega_\theta dx}_{K_4}.
\end{aligned}  
\end{equation}
Then, using the H\"older inequality and the divergence-free of $\bm{u}$, one deduce that
\begin{equation*}
\begin{aligned}
    & |K_1|=\frac{1}{p}\int_{\mathbb R^3}(u_r\p_r+u_z\p_z)(|\omega_\theta|^p) dx = \frac{1}{p}\int_{\mathbb R^3} \bm{u} \cdot\nabla(|\omega_\theta|^p) dx=0,\\
    & |K_2| \leq \|\frac{u_r}{r}\|_{L^\infty}\|\omega_\theta\|_{L^p}^p,\\
    & |K_3| \leq \|\frac{u_\theta}{r} \|_{L^p}\|\omega_r\|_{L^\infty} \|\omega_\theta\|_{L^p}^{p-1},\\
    & |K_4| \leq \|\p_z\mathcal{H}\|_{L^p}\|h_\theta\|_{L^\infty}\|\omega_\theta\|_{L^p}^{p-1}.
\end{aligned}
\end{equation*}
Canceling both sides by $\|\omega_\theta\|_{L^p}^{p-1}$ and integrating over $(0,t)$, one arrives
\begin{equation*}\label{w_L_p}
\|\omega_\theta(t,\cdot)\|_{L^p} \lesssim L_1+ \int_0^t \|\omega_\theta(s,\cdot)\|_{L^p} \|\frac{u_r}{r}(s,\cdot)\|_{L^\infty} ds,
\end{equation*}
here 
\begin{equation*}
    L_1:=\|\omega_{\theta,0}\|_{L^p} + \|\omega_{r}\|_{L^1_t L^{\infty}} \|\frac{u_\theta}{r} \|_{L^\infty_t L^p} + \|\p_z \mathcal{H}\|_{L^1_t L^p} \|h_\theta\|_{L^\infty_t L^\infty}.
\end{equation*}
Based on the aforementioned estimates of $\frac{u_r}{r}$ in \eqref{EURR}, $h_\theta$ in \eqref{ESSS} and $\frac{u_\theta}{r}$ in \eqref{EUTHR}, the estimate of $\nabla \mathcal{H}$ in Proposition \ref{OH2}, together with the a priori assumption \eqref{ineq:prio2}, we have
\begin{equation*}\label{SSSS1}
    L_1 \leq  CE_0^{1/2}\exp \left((1+E_0)(1+t)^{3/2}\exp\big(C(1+E_0)^{4/3}(1+t)^{5/3}\big)\right).
\end{equation*}
Inserting the result into right side of \eqref{eq:wp-basic}, and then by Gr\"onwall inequality, one can deduce
\begin{equation*}
\begin{aligned}
    \|\omega_\theta(t,\cdot)\|_{L^p} &\leq CE_0^{1/2}\exp \left(C(1+E_0)(1+t)^{3/2}\exp\big(C(1+E_0)^{4/3}(1+t)^{5/3}\big)\right) \\
    &\qquad \times \exp(\int_0^t\|\frac{u_r}{r}(s,\cdot)\|_{L^\infty}ds) \\
    & \leq CE_0^{1/2}\exp \left(C(1+E_0)(1+t)^{3/2}\exp\big(C(1+E_0)^{4/3}(1+t)^{5/3}\big)\right), \quad \forall p \in[2,6].
\end{aligned}
\end{equation*}
for any $t \leq T_*$. Furthermore, by Lemma \ref{ineq:nau_b}, we get the $L^p$ estimate of $\nabla \bm{b}$:
\begin{equation*}
    \|\nabla\bm{b}(t,\cd)\|_{L^p} \leq \exp \left(\exp\big(C(1+E_0)^{4/3}(1+t)^{5/3}\big)\right), \quad 2\leq p\leq6, \quad \forall t\leq T_*.
\end{equation*}
This finishes the proof of the proposition.
\end{proof}

\begin{proposition} \label{ineq:na_p_H}
    Assume that $\nabla \cdot \bm{u}_0={h}_{r,0}=h_{z,0} \equiv 0$. Let $(\bm{u}, \bm{h})$, satisfying \eqref{eq:Ax_Hall-MHD}, be the unique local axially symmetric solution of \eqref{eq:Hall-MHD} with the initial data $\left(\bm{u}_0, \bm{h}_0\right) \in H^3\left(\mathbb{R}^3\right)$, the following estimate holds for any $t\geq 0$:
\ba\l{EZZH}
        \|\na\p_z\cal{H}(t,\cd)\|_{L^2}^2+\int_0^t \|\na^2\p_z\cal{H}(s,\cd)\|_{L^2}^2ds\leq \exp\left(\exp\big(C(1+E_0)^{4/3}(1+t)^{5/3}\big)\right).
\ea
\end{proposition}
\begin{proof}
In view of [\cite{li2022single}, (3.25)], we see that
\ba\l{EZH}
    \|\na\p_z\cal{H}(t,\cd)\|_{L^2}^2 & +\int_0^t \|\na^2\p_z\cal{H}(s,\cd)\|_{L^2}^2 ds\\
    & \quad \lesssim \|\na\p_z\cal{H}_0\|_{L^2}+\|\na \bm{b}\|_{L^\infty(0,t;L^2)}\|\na \bm{b}\|_{L^\infty(0,t;L^6)}\int_0^t\|\na^2\cal{H}(s,\cd)\|_{L^2}^2ds.
\ea
Recalling Proposition \ref{OH2} and Proposition \ref{ineq:na_b_Lp}, the following estimates hold:
\begin{equation*}
    \begin{aligned}
        \int_0^t\|\na^2\cal{H}(s,\cd)\|_{L^2}^2ds&\leq \exp\big(C(1+E_0)^{4/3}(1+t)^{5/3}\big);\\
        \|\na \bm{b}\|_{L^\infty(0,t;L^2)}\|\na \bm{b}\|_{L^\infty(0,t;L^6)}&\leq \exp\left(\exp\big(C(1+E_0)^{4/3}(1+t)^{5/3}\big)\right).
    \end{aligned}
\end{equation*}
Inserting above estimates in \eqref{EZH}, one concludes \eqref{EZZH}. This finishes the proof of the proposition.
\end{proof}

\begin{proposition}\label{ineq:p_zH}
    Assume that $\nabla \cdot \bm{u}_0={h}_{r,0}=h_{z,0} \equiv 0$. Let $(\bm{u}, \bm{h})$, satisfying \eqref{ineq:prio2}, be the unique local axially symmetric solution of \eqref{eq:Hall-MHD} with the initial data $\left(\bm{u}_0, \bm{h}_0\right) \in H^3\left(\mathbb{R}^3\right)$, the following estimate of  $\partial_z \cal{H}$ holds for any $t\geq 0$:
    \begin{equation}\label{ineq:p_zH2}
    \begin{aligned}
&\|\p_z\cal{H}(s,\cd)\|_{L^1(0,t;L^\infty)}\leq \exp\left((\exp\big(C(1+E_0)^{4/3}(1+t)^{5/3}\big)\right).
    \end{aligned}
    \end{equation}
\end{proposition}
\begin{proof}
    We know that by using Lemma \ref{ineq:G_N} and the H\"older inequality 
    \begin{equation*}
    \begin{aligned}
         \int_0^t\|\p_z\cal{H}(s,\cd)\|_{L^\infty}ds&\lesssim \int_0^t \|\na\p_z\cal{H}(s,\cd)\|_{L^2}^{1/4}\|\na^2\p_z\cal{H}(s,\cd)\|_{L^2}^{3/4}ds\\
         &\lesssim t^{5/8}\left(\sup_{0\leq s\leq t}\|\na\p_z\cal{H}(s,\cd)\|_{L^2}\right)^{1/4}\left(\int_0^t \|\na^2\p_z\cal{H}(s,\cd)\|_{L^2}^2ds\right)^{3/8}.
    \end{aligned}
    \end{equation*}
    Thus by Proposition \ref{ineq:na_p_H}, one concludes \eqref{ineq:p_zH2}. This finishes the proof of the proposition.
    \end{proof}
    
\begin{proposition}\label{ineq:omega_the}
Assume that $\nabla \cdot \bm{u}_0={h}_{r,0}=h_{z,0} \equiv 0$. Let $(\bm{u}, \bm{h})$, satisfying \eqref{ineq:prio2}, be the unique local axially symmetric solution of \eqref{eq:Hall-MHD} with the initial data $\left(\bm{u}_0, \bm{h}_0\right) \in H^3\left(\mathbb{R}^3\right)$, the following estimate of  $\omega_\theta$ holds for any $t\geq 0$:
\begin{equation}\label{Omage_1}
\|\omega_\theta(t,\cdot)\|_{L^\infty} \leq \exp\left(\exp\big(C(1+E_0)^{4/3}(1+t)^{5/3}\big)\right).
\end{equation}
\end{proposition}

\begin{proof}
Taking the $L^\infty$ estimate of \eqref{eq:w}$_2$, using the a priori assumption  \eqref{ineq:prio2}, the estimate of $\|\f{u_\theta}{r}(t,\cd)\|_{L^p}$ in \eqref{EUTHR} and Proposition \ref{ineq:h_the}, one finds
\begin{equation*}
\begin{aligned}
    \|\omega_\theta(t,\cdot)\|_{L^\infty} &\lesssim \|\omega_{\theta,0}\|_{L^\infty} + \int_0^t\|\frac{u_r}{r}(s,\cdot)\|_{L^\infty}\|\omega_\theta(s,\cdot)\|_{L^\infty}ds\\
    & \q + \|\omega_r(s,\cd)\|_{L^1(0,t;L^\infty)}\|\frac{u_\theta}{r}(s,\cd)\|_{L^\infty(0,t;L^\infty)} + \|\p_z\mathcal{H}(s,\cd)\|_{L^1(0,t;L^\infty)}\|h_\theta(s,\cd)\|_{L^\infty(0,t;L^\infty)}\\
    &\leq \exp\left(\exp\big(C(1+E_0)^{4/3}(1+t)^{5/3}\big)\right) + \int_0^t\|\frac{u_r}{r}(s,\cdot)\|_{L^\infty}\|\omega_\theta(s,\cdot)\|_{L^\infty}ds.
\end{aligned}
\end{equation*}
Thus Gr\"onwall inequality, together with the estimate of $\int_0^t \|\f{u_r}{r}(s,\cd)\|_{L^\infty}ds$ in \eqref{EURR}, indicates
\begin{equation*}
    \|\omega_\theta(t,\cdot)\|_{L^\infty} \leq \exp\left(\exp\big(C(1+E_0)^{4/3}(1+t)^{5/3}\big)\right).
\end{equation*}
This finishes the proof of the proposition.
\end{proof}

Now we are ready to prove next proposition, which is the estimate of $\|\nabla \bm{h}\|_{{L^2}(0,t;L^\infty)}$ and $\|\na^2 \bm{h}\|_{L^2(0,t;L^3)}^2$.

\begin{proposition}\label{ineq:nah_L2}
    Assume that $\nabla \cdot \bm{u}_0 = h_{r,0} = h_{z,0} \equiv 0$. Let $(\bm{u},\bm{h})$, satisfying \eqref{ineq:prio2}, be the unique local axially symmetric solution of \eqref{eq:Hall-MHD} with the initial date $(\bm{u}_0,\bm{h}_0) \in H^3(\mathbb{R}^3) $. Then the magnetic field $\bm{h}$ satisfies the following estimate for any $t\geq 0$:
    \begin{equation}\label{reee}
    \|\nabla \bm{h}\|_{{L^2}(0,t;L^\infty)} + \|\na^2 \bm{h}\|_{L^2(0,t;L^3)}^2 \leq \exp \left(\exp \left(\exp\big(C(1+E_0)^{4/3}(1+t)^{5/3}\big)\right)\right).
\end{equation}
\end{proposition}
\begin{proof}
    First we perform $L^2$ inner product of \eqref{eq:Hall-MHD}$_2$ with $\Delta h$ and integrate by parts to obtain
\begin{equation*}\label{ineq:nah_L2_1}
\begin{aligned}
        & \q \frac{d}{dt}\|\nabla \bm{h}(t,\cdot)\|_{L^2}^2 + \|\nabla^2 \bm{h}(t,\cdot)\|_{L^2}^2\\
        &\leq \underbrace{\int_{\mathbb{R}^3} \left|\left(\frac{h_\theta u_r}{r} - (u_r\p_r + u_z\p_z)h_\theta\right)\bm{e_\theta} \right||\Delta \bm{h}|dx}_{Q_1} + \underbrace{\int_{\mathbb{R}^3}\left|\frac{\p_z(h_\theta)^2}{r}\bm{e_\theta}\right||\Delta \bm{h}|dx}_{Q_2}.
\end{aligned}
\end{equation*}
Because of the previous $L^2$ bound of $\bm{b}$ and $\nabla \bm{b}$ in \eqref{ineq:nau_b} and Proposition \ref{ineq:na_b_Lp}, the Cauchy inequality and interpolation tell us
\begin{equation}\label{ineq:Q1}
\begin{aligned}
    Q_1 &\leq \left(\|\frac{h_\theta u_r}{r}(t,\cdot)\|_{L^2} + \|\bm{b}\cdot\nabla h_\theta(t,\cdot)\|_{L^2}\right)\|\Delta \bm{h}(t,\cdot)\|_{L^2}\\
    &\lesssim \|\mathcal{H}_0\|_{L^{3}}^2\|\bm{b}(t,\cdot)\|_{L^6}\|\nabla^2 \bm{h}(t,\cdot)\|_{L^2} + \|\nabla \bm{b}(t,\cdot)\|_{L^2}\|\nabla \bm{h}(t,\cdot)\|_{L^2}^{\frac{1}{2}}\|\nabla^2 \bm{h}(t,\cdot)\|_{L^2}^{\frac{3}{2}}\\
    &\leq \frac{1}{4}\|\nabla^2 \bm{h}(t,\cdot)\|_{L^2}^2 + C\Big(\|\cal{H}_0\|_{L^3} ^ 2 + \|\nabla \bm{h}(t,\cdot)\|_{L^2}^2\|\na\bm{b}(t,\cd)\|_{L^2}^2\Big)\|\nabla\bm{b}(t,\cdot)\|_{L^2}^2.
\end{aligned}
\end{equation}
Meanwhile, the term $Q_2$ can be estimates exactly as $Q_1$, which follows that
\begin{equation}\label{ineq:Q2}
\begin{aligned}
    Q_2 &\leq 2\|\mathcal{H}_0\|_{L^{\infty}}\|\nabla \bm{h}(t,\cdot)\|_{L^2}\|\Delta \bm{h}(t,\cdot)\|_{L^2}\leq \frac{1}{4}\|\nabla^2 \bm{h}(t,\cdot)\|_{L^2}^2 + \|\cal{H}_0\|_{L^\infty}^2\|\nabla \bm{h}(t,\cdot)\|_{L^2}^2.
\end{aligned}
\end{equation}
Subsituting \eqref{ineq:Q1} and \eqref{ineq:Q2} into \eqref{ineq:na_H_L2}, we have
\begin{equation*}\label{es_nh2}
\begin{aligned}
    \q \frac{d}{dt}\|\nabla \bm{h}(t,\cdot)\|_{L^2}^2 + \|\nabla^2 \bm{h}(t,\cdot)\|_{L^2}^2 \leq& C\Big(\|\cal{H}_0\|_{L^3} ^ 2 + \|\nabla \bm{h}(t,\cdot)\|_{L^2}^2\|\na\bm{b}(t,\cd)\|_{L^2}^2\Big)\|\nabla\bm{b}(t,\cdot)\|_{L^2}^2 \\
   & + \|\cal{H}_0\|_{L^\infty}^2\|\nabla \bm{h}(t,\cdot)\|_{L^2}^2.
\end{aligned}
\end{equation*}
Then using the Gr\"onwall inequality, one concludes 
\begin{equation}\label{es_nh3}
\begin{aligned}
      \|\nabla \bm{h}(t,\cdot)\|_{L^2}^2 + \int_{0}^{t}\|\nabla^2 \bm{h}(s,\cdot)\|_{L^2}^2ds & \leq \Big(\|\na\bm{h}_0\|_{L^2}^2 + \|\cal{H}_0\|_{L^3}^2 \int_0^t \|\na\bm{b}(s,\cd)\|_{L^2}^2ds \Big)\\
    & \qquad \times\exp\Big(\int_{0}^{t}(\|\nabla\bm{b}(s,\cdot)\|_{L^2}^4 + \|\cal{H}_0\|_{L^\infty}^2 )ds\Big) \\
     & \leq \exp \left(\exp \left(\exp\big(C(1+E_0)^{4/3}(1+t)^{5/3}\big)\right)\right).
\end{aligned}
\end{equation}
Based on this $L^2\cap \dot H^1$ bound of $\nabla \bm{h}$, now we are in a position to derive a higher order regularity of $\bm{h}$ by the maximal regularity of heat flow. By direct computation from \eqref{eq:Hall-MHD}$_2$ and \eqref{eq:Ax_Hall-MHD}$_4$, one derives
\begin{equation*}
    \p_t\nabla \bm{h} - \Delta\nabla \bm{h} = \nabla\big(f(r,z)\bm{e_\theta}\big),
\end{equation*}
where
\begin{equation*}
    f(t,r,z) = \frac{h_\theta u_r}{r} - (u_r\p_r + u_z\p_z)h_\theta + \frac{\p_z(h_\theta)^2}{r}.
\end{equation*}
\begin{equation*}
    \nabla\big(f(r,z)\bm{e_\theta}\big) = \p_rf\bm{e_\theta}\otimes\bm{e_r} - \frac{f}{r}\bm{e_r}\otimes\bm{e_\theta} + \p_zf\bm{e_\theta}\otimes\bm{e_z}.
\end{equation*}
By the Duhamel formula, we have
\[
\na\bm{h}(t)=e^{t\Delta}\na\bm{h}_0 + \int_0^te^{(t-s)\Delta}\na({f}(s)\bm{e_\theta})\,ds.
\]
Hence
\begin{equation*}
    \na^3\bm{h}(t)=\na^2e^{t\Delta}\na\bm{h_0} + \int_0^t \na^2 e^{(t-s)\Delta}\na(f(s)\bm{e_\theta})ds.
\end{equation*}
For the initial data term, by Lemma \ref{BXXL}, with $l=r=2$, $\alpha=0$ and $g=\nabla^3\bm{h_0}$, we have
\[
\|\nabla^2 e^{t\Delta}\nabla \bm{h_0}\|_{L^2}=\|e^{t\Delta}\nabla^3\bm{h_0}\|_{L^2}\le C\|\nabla^3\bm{h_0}\|_{L^2}.
\]
The above inequality was given by Lemma \eqref{ineq:heat}. Therefore,
\[
\|\nabla^2 e^{t\Delta}\nabla \bm{h_0}\|_{L^2(0,t;L^2)}
\le Ct^{1/2}\|\nabla^3\bm{h_0}\|_{L^2}\le Ct^{1/2}E_0^{1/2}.
\]
Together with the above inequality, by lemma \ref{ineq:heat} one can deduce
\[
\|\na^3\bm{h}\|_{L^2(0,t;L^2)} \leq t^{1/2}E_0^{1/2} + \|\na(f(r,z))\bm{e_\theta}\|_{L^2(0,t;L^2)}.
\]
For the second term on the right hand side, we begin with the estimate of $\p_rf$. Noting that
\begin{equation*}
    \p_rf = \underbrace{\p_rh_\theta\frac{u_r}{r} - \p_r\bm{b}\cdot\nabla h_\theta}_{|\nabla\bm{b}||\nabla \bm{h}|} + \underbrace{\mathcal{H}\p_ru_r - \mathcal{H}\frac{u_r}{r}}_{|\mathcal{H}||\nabla\bm{b}|} -\underbrace{\bm{b}\cdot\nabla\p_rh_\theta}_{|\bm{b}||\nabla^2 \bm{h}|} + 2\underbrace{\p_{rz}^2h_\theta\mathcal{H}}_{|\mathcal{H}||\nabla^2 \bm{h}|} + 2\underbrace{\p_z\mathcal{H}\p_rh_\theta}_{|\p_z\mathcal{H}||\nabla \bm{h}|} - 2\underbrace{\mathcal{H}\p_z\mathcal{H}}_{|\mathcal{H}||\p_z\mathcal{H}|}.
\end{equation*}
Previous result, including Lemma \ref{ineq:H_Lp},  Proposition \ref{OH2}, Proposition \ref{ineq:na_b_Lp}, Proposition \ref{ineq:na_p_H}, together with \eqref{es_nh3}, indicate that
\begin{equation*}
\begin{aligned}
\|\mathcal{H}(s,\cd)\|_{L^\infty(0,t;{L^2\cap L^\infty})} &\leq E_0,\\
\|\na\bm{b}(s,\cd)\|_{L^\infty(0,t;{L^2\cap L^6})} &\leq \exp \left(\exp\big(C(1+E_0)^{4/3}(1+t)^{5/3}\big)\right),\\
\|\na\mathcal{H}(s,\cd)\|_{L^\infty(0,t;{L^2})} &\leq \exp\big(C(1+E_0)^{4/3}(1+t)^{5/3}\big),\\
\|\na\p_z\mathcal{H}(s,\cd)\|_{L^\infty(0,t;{L^2})}&\leq \exp\left(\exp\big(C(1+E_0)^{4/3}(1+t)^{5/3}\big)\right),\\
\|\na \bm{h}(s,\cd)\|_{{L^\infty(0,t;{L^2})}\cap {L^2(0,t;{\dot{H}^1)}}} &\leq \exp \left(\exp \left(\exp\big(C(1+E_0)^{4/3}(1+t)^{5/3}\big)\right)\right).
\end{aligned}
\end{equation*}
Using Lemma \ref{ineq:G_N} and the H\"older inequality, it is clear that
\begin{equation}\label{ineq:p_r_f}
\begin{aligned}
    \|\p_rf\|_{L^2}^2 &\lesssim \|\nabla \bm{b}\nabla \bm{h}\|_{{L^2}}^2 + \|\mathcal{H}\nabla\bm{b}\|_{{L^2}}^2 + \|\bm{b}\nabla^2 \bm{h}\|_{{L^2}}^2 + \|\mathcal{H}\nabla^2 \bm{h}\|_{{L^2}}^2 + \|\p_z\mathcal{H}\nabla \bm{h}\|_{{L^2}}^2 + \|\mathcal{H}\p_z\mathcal{H}\|_{{L^2}}^2\\
    &\lesssim \|\nabla\bm{b}\|_{L^6}^2\|\nabla \bm{h}\|_{L^2}\|\nabla^2 \bm{h}\|_{L^2} + \|\mathcal{H}\|_{L^{\infty}}^2\|\nabla \bm{b}\|_{L^2}^2 + \|\nabla\bm{b}\|_{L^6}^2\|\nabla^2\bm{h}\|_{L^2}^2 + \|\mathcal{H}\|_{L^{\infty}}^2\|\nabla^2\bm{h}\|_{L^2}^2 \\
    &\q + \|\p_z\mathcal{H}\|_{L^{2}}\|\na\p_z\mathcal{H}\|_{L^2}\|\nabla^2 \bm{h}\|_{L^2}
    ^2 + \|\mathcal{H}\|_{L^{\infty}}^2\|\p_z\mathcal{H}\|_{L^2}^2.
\end{aligned}
\end{equation}
Thus one can deduce that 
\begin{equation*}
    \|\p_r f\|_{{L^2(0,t;L^2)}} \leq \exp \left(\exp \left(\exp\big(C(1+E_0)^{4/3}(1+t)^{5/3}\big)\right)\right).
\end{equation*}
Moreover, by adapting the above argument, one can further derive that 
\begin{equation}\label{ineq:p_z_f}
    \|\frac{f}{r}\|_{{L^2(0,t;L^2)}} + \|\p_zf\|_{{L^2(0,t;L^2)}} \leq \exp \left(\exp \left(\exp\big(C(1+E_0)^{4/3}(1+t)^{5/3}\big)\right)\right).
\end{equation}
Thus, together with \eqref{ineq:p_r_f} and \eqref{ineq:p_z_f}, one can deduce
\begin{equation*}
    \|\nabla\big(f(r,z)\bm{e_\theta}\big)\|_{{L^2(0,t;L^2)}} \leq \exp \left(\exp \left(\exp\big(C(1+E_0)^{4/3}(1+t)^{5/3}\big)\right)\right).
\end{equation*}

Using Lemma \ref{ineq:heat}, one infers that by Sobolev embedding
\begin{equation*}
\begin{aligned}
     \|\nabla^3 \bm{h}\|_{{L^2}(0,t;L^2)} &\lesssim  t^{1/2}E_0^{1/2}+\|\nabla\big(f(r,z)\bm{e_\theta}\big)\|_{{L^2}(0,t;L^2)} \\
     & \leq \exp \left(\exp \left(\exp\big(C(1+E_0)^{4/3}(1+t)^{5/3}\big)\right)\right),
\end{aligned}
\end{equation*}
which indicates
\begin{equation}\label{werr}
    \|\nabla \bm{h}\|_{{L^2}(0,t;L^\infty)} \leq \exp \left(\exp \left(\exp\big(C(1+E_0)^{4/3}(1+t)^{5/3}\big)\right)\right).
\end{equation}
and
\begin{equation}\label{werr2}
\begin{aligned}
    \int_0^t\|\na^2 \bm{h}\|_{L^3}^2dt &\leq C\int_0^t\|\na ^3\bm{h}\|_{L^2}\|\na^2\bm{h}\|_{L^2}dt \leq C\left(\int_0^t\|\na ^3\bm{h}\|_{L^2}^2dt\right)^\frac12\left(\int_0^t\|\na ^2\bm{h}\|_{L^2}^2dt\right)^\frac12 \\
    &\leq \exp \left(\exp \left(\exp\big(C(1+E_0)^{4/3}(1+t)^{5/3}\big)\right)\right).
\end{aligned}
\end{equation}
Adding \eqref{werr} and \eqref{werr2} together, one concludes \eqref{reee}. This finishes the proof in the proposition.
\end{proof}

After the above preparations, we are now in a position to focus on the higher-order estimates.
\begin{proposition}\label{ineq:High2}
Under the same condition as Theorem \ref{theo}. let $(\bm{u},\bm{h})$ be a smooth axially symmetric solution on $[0,T_*]$ of \eqref{eq:Hall-MHD}, then 
\begin{equation*}
    e+E(t) \leq \left( (e+E_0)\exp \left(\int_0^t (1+ \|\na^2\bm{h}(s,\cd)\|_{L^3}^2 +\|\na\bm{h}(s,\cd)\|_{L^\infty}^2)ds\right)\right)^{\exp\left(\int_0^t \|\na\times\bm{u}(s,\cd)\|_{L^\infty}ds\right)}.
\end{equation*}
\end{proposition}

\begin{proof}
    Define $E(t):=\|(\bm{u},\bm{h})(t,\cd)\|_{H^3}^2$, applying $\na^3$ to the equations of \eqref{eq:Hall-MHD} respectively, we have
\begin{equation*}
    \begin{aligned}
&\p_t\na^3\bm{u}+\na^3(\bm{u}\cd\na\bm{u})+\na^4p=\na^3(\bm{h}\cd\na\bm{h}),\\
        &\p_t \na^3\bm{h}+\na^3(\bm{u}\cd\na\bm{h})+\na^3\na\times((\na\times \bm{h})\times \bm{h})= \na^3(\bm{h}\cd\na\bm{u})+\na^3\Delta\bm{h}.
    \end{aligned}
\end{equation*}
Integrating on $L^2$ and adding them together, we have
\begin{equation}\label{eq:high_ord}
    \f{1}{2}\f{d}{dt}\|\na^3(\bm{u},\bm{h)} (t,\cd)\|_{L^2}^2 +\|\na^4\bm{h}(t,\cd)\|_{L^2}^2 = N_1+N_2+N_3+N_4+N_5.
\end{equation}
Here
\begin{equation*}
    \begin{aligned}
        &N_1:=-\int_{\bb{R}^3}\na^3(\bm{u}\cd\na\bm{u})\cd\na^3\bm{u}dx,\\
        &N_2:=-\int_{\bb{R}^3}\na^3(\bm{u}\cd\na\bm{h})\cd\na^3\bm{h}dx,\\
        &N_3:=\int_{\bb{R}^3} \na^3(\bm{h}\cd\na\bm{h})\cd\na^3\bm{u}dx,\\
        &N_4:=\int_{\bb{R}^3}\na^3(\bm{h}\cd\na\bm{u})\cd\na^3\bm{h}dx,\\
        &N_5:=-\int_{\bb{R}^3}\na\times\na^3((\nabla\times\bm{h})\times\bm{h})\cd\na^3\bm{h}dx.
    \end{aligned}
\end{equation*}
the term $\|\na^4\bm{h}(s,\cd)\|_{L^2}^2$ appears due to
\begin{equation*}
    \int_{\bb{R}^3}\na^3\Delta\bm{h}\cd\na^3\bm{h}dx=-\|\na^4\bm{h}\|_{L^2}^2.
\end{equation*}
Noting that
\begin{equation*}
\begin{aligned}
     &|N_1| = \left| -\int_{\bb{R}^3}\bm{u}\cd\na\na^3\bm{u}dx+\int_{\bb{R}^3}[\na^3,\bm{u}\cd\na]\bm{u}dx\right|=\left|\int_{\bb{R}^3}[\na^3,\bm{u}\cd\na]\bm{u}dx\right|\leq \|\na\bm{u}\|_{L^\infty}\|\na^3\bm{u}\|_{L^2}^2,\\
     &|N_2|=\left|-\int_{\bb{R}^3}\bm{u}\cd\na\na^3\bm{h}\cd\na^3\bm{h}dx-\int_{\bb{R}^3}[\na^3,\bm{u}\cd\na]\bm{h}\cd\na^3\bm{h}dx\right|\\
     &\qquad =\left|-\int_{\bb{R}^3}[\na^3,\bm{u}\cd\na]\bm{h}\cd\na^3\bm{h}dx\right|\leq \|\na\bm{u}\|_{L^\infty}\|\na^3\bm{h}\|_{L^2}^2.\\
\end{aligned}
\end{equation*}
The first term on the right of the above equalities vanishe by integrate by parts and the fact $\na\cd\bm{u}=0$. Adding them together, one can deduce
\begin{equation}\label{ineq:N1N2}
    |N_1|+|N_2|\leq C\|\na\bm{u}\|_{L^\infty}(\|\na^3\bm{u}\|_{L^2}^2+\|\na^3\bm{h}\|_{L^2}^2).
\end{equation}
Similarly, by using commutator notation, one can also write
\begin{equation}
\begin{aligned}\label{eq:N3N4}
    \left|N_3+N_4\right| &=\Big|\int_{\bb{R}^3}\bm{h}\cd\na\na^3\bm{h}\cd\na^3\bm{u}dx + \int_{\bb{R}^3}\bm{h}\cd\na\na^3\bm{u}\cd\na^3\bm{h}dx + \int_{\bb{R}^3}[\na^3,\bm{h}\cd\na]\bm{h}\cd\na^3\bm{u}dx\\
    &\qquad +\int_{\bb{R}^3}[\na^3,\bm{h}\cd\na]\bm{u}\cd\na^3\bm{h}dx\Big|.
\end{aligned}
\end{equation}
Noting that
\begin{equation*}
\int_{\bb{R}^3}\bm{h}\cd\na\na^3\bm{h}\cd\na^3\bm{h}ds+\int_{\bb{R}^3}\bm{h}\cd\na\na^3\bm{u}\cd\na^3\bm{h}dx = \int_{\bb{R}^3} h\cd\na(\na^3\bm{h}\cd\na^3\bm{u})=0.
\end{equation*}
Thus \eqref{eq:N3N4} satisfy by Lemma \ref{ineq:commu}:
\begin{equation}\label{ineq:N3N4}
    \left|N_3 + N_4\right|\leq C\|\na\bm{h}\|_{L^\infty}(\|\na^3\bm{u}\|_{L^2}^2+\|\na^3\bm{h}\|_{L^2}^2) .
\end{equation}
Lastly, we estimate $N_5$.
\begin{equation*}
\begin{aligned}
    N_5 &=  \left|\int_{\bb{R}^3} \na^3((\na\times \bm{h})\times \bm{h} ) \cd \na^3(\na\times \bm{h})dx\right|\\
    & =\left|\int_{\bb{R}^3} \big(\na^3((\na\times \bm{h})\times \bm{h} )- \na^3(\na\times\bm{h})\times \bm{h}\big) \cd \na^3(\na\times \bm{h})dx \right|\\
    & \leq \int_{\bb{R}^3} \left| \big(\na^3((\na\times \bm{h})\times \bm{h} )- \na^3(\na\times\bm{h})\times \bm{h}\big)\right| \left| \na^3(\na\times \bm{h})\right|dx\\
    & \leq  \int_{\bb{R}^3} \left| \big(\na^3((\na\times \bm{h})\times \bm{h} )- \na^3(\na\times\bm{h})\times \bm{h}\big)\right| \left| \na^4 \bm{h}\right|dx.
\end{aligned}
\end{equation*}
Here the second equality follows from the fact
\[
(\na^3(\na\times\bm{h})\times\bm{h})\cd\na^3(\na\times\bm{h})=0.
\]
Using the Leibniz formula,
\begin{equation*}
    |\big(\na^3((\na\times \bm{h})\times \bm{h} )- \na^3(\na\times\bm{h})\times \bm{h}\big) |  \leq C (|\na\bm{h}||\na^3\bm{h}|+ |\na^2\bm{h}|^2),
\end{equation*}
thus one deduces
\begin{equation}\label{ineq:N5}
\begin{aligned}
     |N_5|& \leq C \int_{\bb{R}^3} \left(|\na^2\bm{h}|^2|\na^4\bm{h}| + |\na \bm{h}||\na^3\bm{h}||\na^4\bm{h}|\right)dx\\
    & \leq C \|\na^4 \bm{h}\|_{L^2}\left(\|\na\bm{h}\|_{L^\infty}\|\na^3\bm{h}\|_{L^2}+\|\na^2\bm{h}\|_{L^3}\|\na^2\bm{h}\|_{L^6}\right)\\
    & \leq C    \|\na^4\bm{h}\|_{L^2}\|\na^3\bm{h}\|_{L^2}\left(\|\na\bm{h}\|_{L^\infty}+ \|\na^2\bm{h}\|_{L^3}\right)\\
    & \leq \f12\|\na^4\bm{h}\|_{L^2}^2 + C\|\na^3\bm{h}\|_{L^2}^2(\|\na\bm{h}\|_{L^\infty}^2 + \|\na^2\bm{h}\|_{L^3}^2).
\end{aligned}
\end{equation}
The second inequality above is a consequence of the H\"older inequality, the third follows from the Sobolev embedding theorem in $\mathbb{R}^3$ and the last follows from Young's inqaulity.

Adding \eqref{ineq:N1N2}, \eqref{ineq:N3N4} and \eqref{ineq:N5} together, inserting the result into the right side of \eqref{eq:high_ord}, one arrives
\begin{equation}\label{ineq:High1}
\begin{aligned}
    \f{d}{dt}(e+\|\na^3(\bm{u},\bm{h)} (t,\cd)\|_{L^2}^2) &\leq C\left(\|\na\bm{u}(t,\cd)\|_{L^\infty} + \|\na^2\bm{h}(t,\cd)\|_{L^3}^2  + \|\na\bm{h}(t,\cd)\|_{L^\infty}^2\right)\\
    &\qquad\times \left(e+\|\na^3(\bm{u},\bm{h)} (t,\cd)\|_{L^2}^2\right).
\end{aligned}
\end{equation}

By Lemma \ref{ineq:BMO}, one can deduce
\begin{equation*}
\begin{aligned}
	\|\nabla \boldsymbol{u}(t,\cd)\|_{L^\infty} & \lesssim 1+ \|\nabla \times \boldsymbol{u}(t,\cd)\|_{BMO(\mathbb{R}^3)} \log \left(e+\|\boldsymbol{u}(t,\cd)\|_{H^3}\right)\\
	& \lesssim 1 +  \|\na \times\bm{u}(t,\cd)\|_{L^\infty}\log(e+\|\bm{u}(t,\cd)\|_{H^3}).
\end{aligned}
\end{equation*}

Integrating the inequality \eqref{ineq:High1} over $(0,t)$ and together with the above inequality, one can deduce that 
\begin{equation*}
\begin{aligned}
    e+E(t)\leq e+E_0 +\int_0^t \left(1+\|\na\times\bm{u}(s,\cd)\|_{L^\infty}\log (e+E(s))+ \|\na^2\bm{h}(t,\cd)\|_{L^3}^2 +  \|\na\bm{h}(s,\cd)\|_{L^\infty}^2\right)\\
    \times \left(e+E(s)\right)ds.
\end{aligned}
\end{equation*}
Using the Gr\"onwall inequality, one arrives
\begin{equation*}
\begin{aligned}
    e+E(t)\leq (e+E_0)&\exp\left(\int_0^t(1+ \|\na^2\bm{h}(t,\cd)\|_{L^3}^2 +\|\na\bm{h}(s,\cd)\|_{L^\infty}^2)ds \right.\\
     &\left.\qquad \qquad \qquad\qquad\qquad + \int_0^t\|\na\times\bm{u}(s,\cd)\|_{L^\infty}\log(e+E(s))ds\right).
\end{aligned}
\end{equation*}
Next, we take the logarithm on both sides of the above inequality to get
\begin{equation*}
\begin{aligned}
    \log(e+E(t))\leq \log(e+E_0) &+ \int_0^t\left(1+\|\na^2\bm{h}(s,\cd)\|_{L^3}^2+\|\na\bm{h}(s,\cd)\|^2_{L^\infty}\right)ds \\
    &+  \int_0^t  \|\na\times\bm{u}(s,\cd)\|_{L^\infty}\log(e+E(s))ds.
\end{aligned}
\end{equation*}
Finally, applying the Gr\"onwall inequality again, one arrives
\begin{equation*}
\begin{aligned}
    \log(e+E(t))\leq & \left(\log(e+E_0)+ \int_0^t (1 + \|\na^2\bm{h}(s,\cd)\|_{L^3}^2  +\|\na\bm{h}(s,\cd)\|_{L^\infty}^2)ds\right)\\
     &\times\exp\left(\int_0^t \|\na\times\bm{u}(s,\cd)\|_{L^\infty}ds\right).
\end{aligned}
\end{equation*}
Then taking the exponential on both sides of above inequality, one arrives
\begin{equation*}
    e+E(t) \leq \left( (e+E_0)\exp \left(\int_0^t (1+ \|\na^2\bm{h}(s,\cd)\|_{L^3}^2 +\|\na\bm{h}(s,\cd)\|_{L^\infty}^2)ds\right)\right)^{\exp\left(\int_0^t \|\na\times\bm{u}(s,\cd)\|_{L^\infty}ds\right)}.
\end{equation*}
\end{proof}

\section{The end of the proof}\label{endd}
Now it remains to verify the a priori assumption \eqref{ineq:prio2}. We begin by investigating $\eqref{eq:w}_{1,3}$, 
\begin{equation*} 
\left\{\begin{aligned}
	& \p_t \omega_r + (u_r\p_r + u_z\p_z)\omega_r = (\omega_r\p_r + \omega_z\p_z)u_r,\\
	& \p_t \omega_{z} + (u_r\p_r + u_z\p_z)\omega_z = (\omega_r\p_r + \omega_z\p_z)u_z.
\end{aligned}\right.
\end{equation*}
Rewriting it in vector form, we obtain
\begin{equation*}
    \p_t \bm{U} + \bm{b} \cdot \na \bm{U} = \bm{A} \cdot \bm{U},
\end{equation*}
which 
\begin{equation*}
    \bm{U} = (\omega_r,\omega_z),\quad  \bm{A}=\begin{pmatrix}
		\p_ru_r  & \p_zu_r\\
        \p_ru_z  & \p_zu_z
	\end{pmatrix}, \quad \bm{b} = u_r\bm{e_r} + u_z\bm{e_z}.
\end{equation*}
Then, by Lemma \ref{ineq:uLp} with $p = \infty$ and identity $\|\bm{A}\|_{L^\infty} \leq \|\na \bm{b}\|_{L^\infty}$, we can deduce
\begin{equation*}
\begin{aligned}
    \|\nabla \times (u_{\theta}\bm{e_{\theta}}) (t,\cdot)\|_{L^\infty} &= \|(\omega_r,\omega_z)(t,\cdot)\|_{L^\infty} \leq  \|\nabla \times (u_{0,\theta}\bm{e_{\theta}})\|_{L^\infty}\exp\left( \int_0^t \|\nabla \bm{b}(s,\cdot)\|_{L^\infty} ds \right)\\
    &\leq \varepsilon\exp \left( \int_0^t (e+E(s))ds\right).
\end{aligned}
\end{equation*}
By Proposition \ref{ineq:High2}, we can obtain the following restriction of $T_*$:
\begin{equation}\label{ineq:fina1}
   \varepsilon T_* \exp\left(\int_0^{T_*}(e+E(s))ds\right)\leq \f{1}{2},
\end{equation}
ensures the a priori estimate \eqref{ineq:prio2}. Noting that \eqref{ineq:fina1} is ensured provided that the following condition holds.
\begin{equation}\label{ineq:fina2}
    \begin{aligned}
        \varepsilon T_*\exp\Bigg(\int_0^{T_*}
        &\Big( (e+E_0)\\
        &\times \exp (\int_0^t (1 + \|\na^2\bm{h}(s,\cd)\|_{L^3}^2 +\|\na\bm{h}(s,\cd)\|_{L^\infty}^2)ds)\Big)^{\exp\left(\int_0^t \|\na\times\bm{u}(s,\cd)\|_{L^\infty}ds\right)}
        \,dt\Bigg)\\
        &\leq \f12.
    \end{aligned}
\end{equation}
Next, we write
\begin{equation*}
\begin{aligned}
    &A(t):=\int_0^{t} \|\na\times \bm{u}(s,\cd)\|_{L^\infty}ds,\\
    &B(t):=\int_0^{t}(1+\|\na^2\bm{h}(s,\cd)\|_{L^3}^2+\|\na\bm{h}(s,\cd)\|^2_{L^\infty})ds.
\end{aligned}
\end{equation*}
Thus \eqref{ineq:fina2} can be rewritten as
\begin{equation} \label{fin2}
    \varepsilon T_*\exp\Bigg( \int_0^{T_*}((e+E_0)e^{B(t)}))^{e^{A(t)}}dt \Bigg)\leq \f12.
\end{equation} 
Since
\begin{equation*}
    (e+E_0)^{e^{A(t)}}\times e^{B(t)\cd e^{A(t)}} = \exp \left( e^{A(t)}(\log(e+E_0)+B(t) \right),
\end{equation*}
then \eqref{fin2} is satisfied if the following estimate holds:
\begin{equation} \label{fina3}
    \varepsilon T_*\exp\left(T_* e^{A(T_*)}\times(\log(e+E_0)+{B(T_*)}) \right)\leq \f12.
\end{equation}
It remains to estimate $A(T_*)$ and $B(T_*)$. By \eqref{Omage_1}, we have the following estimate:
\begin{equation}\label{wwww}
\begin{aligned}
A(T_*) &= \int_0^{T_*} \left(\|(\omega_r,\omega_z)(s,\cd)\|_{L^\infty} + \|\omega_\theta(s,\cd)\|_{L^\infty}\right)ds \\    
& \leq T_*\sup_{0\leq t\leq T_*}\|\na\times u_{\theta}\bm{e_\theta}(t,\cd)\|_{L^\infty} + \int_0^{T_*}\|\omega_{\theta}(t,\cd)\|_{L^\infty}dt\\
&\leq T_*\exp\left(\exp\big(C(1+E_0)^{4/3}(1+T_*)^{5/3}\big)\right).
\end{aligned}
\end{equation}
Using Proposition \ref{ineq:nah_L2}, we have
\begin{equation}\label{ewew}
    B(T_*)\leq T_*\exp \left(\exp \left(\exp\big(C(1+E_0)^{4/3}(1+t)^{5/3}\big)\right)\right).
\end{equation}
Inserting \eqref{wwww} and \eqref{ewew} into the left side of \eqref{fina3}, we have, there exists a constant $C>0$, inequality \eqref{fina3} is satisfied if the following estimate holds:
\begin{equation*}
    \varepsilon\exp\left(\exp(\exp(\exp(C(1+E_0)^{4/3}(1+T_*)^{5/3}))))\right)\leq \f12.
\end{equation*}
Therefore, the condition \eqref{ineq:prio2} is satisfied if we choose
\[
T_* = \f{1}{C(1+E_0)^{4/3}}\left(\log\log\log\log{({2\varepsilon})^{-1}}\right)^{3/5}-1.
\]
Thus for $\varepsilon$ being sufficiently small, there exists $C_*>0$ such that
\begin{equation*}
    T_*= \frac{C_*}{({1+E_0})^{4/5}}\left(\log\log\log\log(\varepsilon^{-1})\right)^{3/5}.
\end{equation*}
This gives validity to the equality \eqref{res} and thus completes the proof of theorem \ref{theo}.

{
\section{Estimate of the vanishing resistivity}\label{ploi}
In the previous sections, we normalized the resistive coefficient by taking
$\nu=1$ for the sake of convenience. However, we point out here that the method developed in this paper breaks down as $\nu\to 0_+$. Indeed, for the case $\nu=0$, one may need the initial datum of $\bm{h}$ to be sufficiently small in a suitable sense in order to obtain a long lifespan result. See \cite{MR5040962} for example. 

For the further study of vanishing resistivity $\nu\to 0^+$, let us briefly explain how the
lifespan estimate \eqref{res} depends exactly on $\nu$. We first give the elementary scaling consequence for the heat operator.
\begin{lemma}\label{PPOe}
 Let $v_0\in H^2$ and $g\in L^2(0,T;L^2)$, and $v$ satisfy the initial value problem
\begin{equation*}
\left\{
\begin{aligned}
&\p_t v - \nu \Delta v=g(t,x)\,,\\
&v(0,x)= v_0(x)\,.
\end{aligned}
\right.
\end{equation*}
Then $v$ satisfies
\ba\l{INVH}
\|\na^2 v\|_{L^2(0,T;L^2)}\leq C\left(T^{\f{1}{2}}\|\na^2v_0\|_{L^2}+\nu^{-1}\|g\|_{L^2(0,T;L^2)}\right)\,.
\ea
\end{lemma}
}
\begin{proof}
Denoting $v=v_1+v_2$, we can split the problem \eqref{INVH} to
\begin{equation}\label{SPLIT}
\left\{
\begin{aligned}
&\p_t v_1 - \nu \Delta v_1=g(t,x)\,,\\
&v_1(0,x)= 0\,;
\end{aligned}
\right.\q\q\left\{
\begin{aligned}
&\p_t v_2 - \nu \Delta v_2=0\,,\\
&v_2(0,x)= v_0(x)\,.
\end{aligned}
\right.
\end{equation}
{Setting $\tau:=\nu t$ and $G(\tau,x):= v_1(\f{\tau}{\nu},x)$, one deduces from \eqref{SPLIT}$_1$ that
\[
\left\{
\begin{aligned}
&\p_\tau G(\tau,x)-\Delta G(\tau,x)=\nu^{-1}g(\tau/\nu,x)\,,\\
&G(0,x)=0\,.
\end{aligned}
\right.
\]
Therefore, by the maximal regularity of heat flow in Lemma \ref{ineq:heat},
\begin{equation}\label{OOw}
\begin{aligned}
\|\na^2 v_1\|_{L^2(0,T;L^2)}
&=\nu^{-1/2}\|\na^2 G\|_{L^2(0,\nu T;L^2)}\\
&\le C\nu^{-1/2}
  \left\|\nu^{-1}g(\tau/\nu,x)\right\|_{L^2(0, \nu T;L^2)}\\
&=C\nu^{-1}\|g\|_{L^2(0,T;L^2)} .
\end{aligned}    
\end{equation}}
For the estimate of $\eqref{SPLIT}_2$, by Lemma \ref{BXXL} with $l=r=2$, $\alpha=0$, one can deduce 
\begin{equation}\label{OOW2}
    \|\na^2v_2\|_{L^2(0,T;L^2)}\leq CT^{1/2}\|\na^2v_0\|_{L^2}.
\end{equation}
Adding \eqref{OOw} and \eqref{OOW2} together, one concludes \eqref{INVH}.
\end{proof}

In the following, we only track the dependence of the constants on the magnetic resistivity coefficient $\nu$.
First, repeating the proofs of Propositions \ref{OH2} and \ref{ineq:na_p_H}, while
keeping the coefficient $\nu$ yields
\begin{equation*}
    \begin{aligned}
        &\sup _{0 \leq s \leq t}\left(\|\Omega(s, \cdot)\|_{L^2 \cap L^6}^2+\|\nabla \cal{H}(s, \cdot)\|_{L^2}^2\right)+\nu\int_0^{t}\left\|\nabla^2 \cal{H}(s, \cdot)\right\|_{L^2}^2 d s\\
         & \qquad \qquad \leq \exp\big(C(1+\nu^{-1})(1+E_0)^{4/3}(1+t)^{5/3}\big).
        \end{aligned}
    \end{equation*}
and
\ba\label{M2}
\|\na\p_z\cal{H}(t,\cd)\|_{L^2}^2+\nu\int_0^t \|\na^2\p_z\cal{H}(s,\cd)\|_{L^2}^2ds\leq \exp\left(\exp\big(C(1+\nu^{-{1}})(1+E_0)^{4/3}(1+t)^{5/3}\big)\right).
\ea

Owing to \eqref{M2}, repeating the proof of Proposition \ref{ineq:p_zH} while keeping $\nu$, we have
\begin{equation}\label{M3}
    \begin{aligned}
&\|\p_z\cal{H}(s,\cd)\|_{L^1(0,t;L^\infty)}\leq \exp\left(\exp\big(C(1+\nu^{-{1}})(1+E_0)^{4/3}(1+t)^{5/3}\big)\right).
\end{aligned}
\end{equation}
With the help of Lemma \ref{PPOe}, similarly as Proposition \ref{ineq:nah_L2}, one can deduce
    \begin{equation*}\label{opo}
    \|\nabla \bm{h}\|_{{L^2}(0,t;L^\infty)} + \|\na^2 \bm{h}\|_{L^2(0,t;L^3)}^2 \leq \exp \left(\exp \left(\exp\big(C(1+\nu^{-1})(1+E_0)^{4/3}(1+t)^{5/3}\big)\right)\right).
\end{equation*}
Combining with the above result and the higher-order estimate in Proposition \ref{ineq:High2}, then applying the same method in Section 4, one can deduce that
\begin{equation*}
    T_*= \f{C(1+\nu^{-1})^{-3/5}}{(1+E_0)^{4/5}}
\left(\log\log\log\log(\varepsilon^{-1})\right)^{3/5}.
\end{equation*}
This complete the proof.

\end{document}